\documentclass[11pt]{amsart}
\usepackage[UKenglish]{isodate}
\usepackage{amsmath}
\usepackage{amsthm}
\usepackage{amsbsy}
\usepackage{amssymb}
\usepackage{amsfonts}
\usepackage[dvips]{lscape}
\usepackage{xcolor}
\usepackage{amscd}
\usepackage[all,cmtip]{xy}
\usepackage{euscript}
\usepackage{parskip}
\usepackage{enumerate}
\usepackage{graphicx}

\usepackage[margin=0.75in]{geometry}
\usepackage[expansion=false]{microtype}

\usepackage[pdfusetitle,unicode,hidelinks]{hyperref}

\theoremstyle{plain}
\newtheorem{theorem}{Theorem}[section]

\newtheorem{proposition}[theorem]{Proposition}
\newtheorem{lemma}[theorem]{Lemma}
\newtheorem{corollary}[theorem]{Corollary}
\newtheorem{remark}[theorem]{Remark}
\newtheorem{definition}[theorem]{Definition}

\newcommand{\Mod}{\cal{M}\textup{od}}

\newcommand{\Sp}{\cal{S}\textup{pt}}
\newcommand{\cal}{\EuScript}
\newcommand{\Spec}{\textup{Spec}}
\newcommand{\Top}{\cal{T}\textup{op}}
\newcommand{\Sm}{\cal{S}\textup{m}}
\newcommand{\Shv}{\cal{S}\textup{hv}}
\newcommand{\PSh}{\cal{P}\textup{Sh}}
\newcommand{\Spc}{\cal{S}\textup{pc}}
\newcommand{\Spt}{\cal{S}\textup{pt}}
\newcommand{\Map}{\textup{Map}}
\newcommand{\et}{\textup{\'et}}

\newcommand{\Ext}{\textup{Ext}}
\newcommand{\Nis}{\textup{Nis}}
\newcommand{\Hom}{\textup{Hom}}
\newcommand{\Pro}{\textup{Pro}}

\newcommand{\colim}{\textup{colim}}
\begin{document}
\title{Comparison of Stable Homotopy Categories and a Generalized Suslin-Voevodsky Theorem}
\author{Masoud Zargar}
\renewcommand{\contentsname}{}
\begin{abstract}
Let $k$ be an algebraically closed field of exponential characteristic $p$. Given any prime $\ell\neq p$, we construct a stable \'etale realization functor
$$\underline{\textup{\'Et}}_{\ell}:\Spt(k)\rightarrow \Pro(\Sp)^{H\mathbb{Z}/\ell}$$
from the stable $\infty$-category of motivic $\mathbb{P}^1$-spectra over $k$ to the stable $\infty$-category of $(H\mathbb{Z}/\ell)^*$-local pro-spectra (see section~\ref{prohom} for definition). This is induced by the \'etale topological realization functor \'a la Friedlander. The constant presheaf functor naturally induces the functor
\[\textup{SH}[1/p]\rightarrow\textup{SH}(k)[1/p],\]
where $k$ and $p$ are as above and $\textup{SH}$ and $\textup{SH}(k)$ are the classical and motivic stable homotopy categories, respectively. We use the stable \'etale realization functor to show that this functor is fully faithful. Furthermore, we conclude with a homotopy theoretic generalization of the \'etale version of the Suslin-Voevodsky theorem.
\end{abstract}
\maketitle
\tableofcontents
\section{Introduction}\label{intro}
Motivic homotopy theory is designed to be the homotopy theory of algebraic geometry. Just as spectra in stable homotopy theory represent generalized cohomology theories, motivic spectra in stable motivic homotopy theory represent \textit{homotopy invariant} algebraic cohomology theories. There are comparison theorems that relate cohomology theories. For example, the Artin comparison theorem gives an isomorphism between the \'etale cohomology $H^*_{et}(X;\mathbb{Z}/n\mathbb{Z})$ of a separated finite type $\mathbb{C}$-scheme $X$ and the singular cohomology $H^*(X^{an};\mathbb{Z}/n\mathbb{Z})$ of its complex analytification $X^{an}$. Another such comparison theorem is the Suslin-Voevodsky theorem establishing an isomorphism between the (motivic) singular (co)homology of a separated finite type $\mathbb{C}$-scheme with finite coefficients and the singular (co)homology of its associated complex analytic space $X^{an}$. Suslin and Voevodsky also show that if $X$ is a separated finite type $k$-scheme, where $k$ is an algebraically closed field of exponential characteristic $p$ \textit{admitting resolution of singularities}, then its (motivic) singular cohomology $H^*_{sing}(X;\mathbb{Z}/n\mathbb{Z})$ with $n$ coprime to $p$ agrees with the \'etale cohomology $H^*_{et}(X;\mathbb{Z}/n\mathbb{Z})$ \cite{SV}. The method of de Jong alterations allows the removal of the condition that $k$ admits resolution of singularities \cite{Geisser}. In combination, these results say that for suitable finite torsion coefficients motivic cohomology is a generalization of both Betti and \'etale cohomology, whenever such cohomology theories are defined and well-behaved. In general, the study of realization functors is important because they, in particular, give us comparison theorems that allow us to understand complicated mathematical objects via their less complicated or concrete shadows. In particular, the existence of realization functors in the study of motives justifies the term \textit{motivic}. In this paper, we construct the stable $\ell$-adic \'etale realization functor, a realization functor on the level of motivic spectra. Using this construction, we relate the classical stable homotopy category to the stable motivic homotopy category and generalize the Suslin-Voevodsky theorem to the homotopy theoretic level.\\
\\
In the classical stable homotopy category $\textup{SH}$, the (topological) circle $S^1$ is inverted, while in the stable motivic homotopy category $\textup{SH}(S)$, the projective line $\mathbb{P}^1_S$ is inverted. Since two different objects are inverted in the two theories, the question arises whether there is a fully faithful functor from $\textup{SH}$ to $\textup{SH}(S)$, i.e. whether stable motivic homotopy theory subsumes classical stable homotopy theory. There has been progress in this respect for various base fields. In \cite{Levine}, Levine used the Betti realization functor on the level of motivic spectra to prove that the functor 
\[\textup{SH}\rightarrow\textup{SH}(k)\]
induced by the constant sheaf functor is fully faithful whenever $k$ is algebraically closed of characteristic zero. The proof given by Levine is by comparing the slice spectral sequence to the Betti realization of the slice spectral sequence. On the other hand, Ostvaer and Wilson \cite{OstWilson} jointly proved that for $k$ algebraically closed of positive characteristic $p$, the functor $\textup{SH}\rightarrow\textup{SH}(k)$ induces an isomorphism $\pi_n(\mathbb{S})[1/p]\rightarrow\underline{\pi}_{n,0}(\mathbf{1}_k)(k)[1/p]$. Their proof passes through Witt vectors to reduce the positive characteristic case to Levine's characteristic zero case. However, they use motivic Serre finiteness (see corollary~\ref{msf} and remark~\ref{msfr}), now a theorem due to Ananyevskiy-Levine-Panin \cite{ALP}, and the motivic Adams spectral sequence to establish their result.\\
\\
On the other hand, we use \'etale homotopy theory to deduce this result and more. Artin and Mazur defined the \'etale homotopy type of a scheme $X$ as a pro-homotopy type \cite{AM}, and this was later rigidified to the notion of a pro-space $\Pi^{\text{\'et}}_{\infty}X$ by Friedlander \cite{Friedlander}. The pro-space $\Pi^{\text{\'et}}_{\infty}X$ is defined such that its singular cohomology and the \'etale cohomology of $X$ agree for constant coefficients. Over a non-separably closed field $k$, it would be better to define $\Pi^{\text{\'et}}_{\infty}X$ \textit{relative} to $\Spec\ k$ so that the action of the absolute Galois group $G_k:=\text{Gal}(k^{sep}|k)$ is also considered; this would give us a pro-sheaf of spaces on the small \'etale site of $\Spec\ k$. The \'etale topological type $X\mapsto \Pi^{\text{\'et}}_{\infty}X$ of smooth $k$-schemes, $k$ an algebraically closed field, has been lifted to the level of motivic spaces \cite{Isaksen1}. In this paper, we will construct a stable version of the $\infty$-categorical incarnation of this construction, that is, for each prime $\ell\neq\text{expchar }k$, we construct a functor
\[\underline{\textup{\'Et}}_{\ell}:\Spt(k)\rightarrow\Pro(\Sp)^{H\mathbb{Z}/\ell}\]
between stable $\infty$-categories. Note that this stable \'etale realization is different from that constructed by Quick in \cite{Quick}. His construction is basically the formal stabilization of the unstable \'etale realization, whereas the one in this paper is not.\\
\\
The first application of our stable $\ell$-adic \'etale realization functor is the generalization of Levine's result to algebraically closed fields that need not be of characteristic zero. Secondly, we generalize the \'etale Suslin-Voevodsky theorem to the homotopy-theoretic level. Indeed, we prove the following theorems.
\begin{theorem}[Theorem~\ref{full}]
Let $k$ be an algebraically closed field of exponential characteristic $p$. Then
\[c[1/p]:\textup{SH}[1/p]\rightarrow\textup{SH}(k)[1/p]\]
induced by the constant presheaf functor is fully faithful.
\end{theorem}

\begin{theorem}[Generalized Suslin-Voevodsky, theorem~\ref{susvoe}]
Suppose $k$ is an algebraically closed field of exponential characteristic $p$, $E\in\Spt^{\textup{eff}}(k)_{\textup{tor}}$ is an effective torsion $\mathbb{P}^1$-spectrum, and $\ell\neq p$ is a prime. Then
\[\underline{\textup{\'Et}}_{\ell\ast}:\underline{\pi}_{n,0}(E)(k)\otimes\mathbb{Z}_{\ell}\rightarrow\pi_n\left(\underline{\textup{\'Et}}_{\ell}(E)\right)\]
is an isomorphism.
\end{theorem}

Though the first theorem can be proved using a density argument in conjunction with the main result of \cite{OstWilson}, the generalized \'etale Suslin-Voevodsky thereom is completely new. That being said, in this paper we give a completely different proof of the first theorem; the author believes that this proof is simpler, more conceptual, and susceptible to further generalization (to non-algebraically closed fields). We deduce the full-faithfulness result by comparing the slice spectral sequence to the spectral sequence induced by the stable \'etale realization of the slice tower. In particular, instead of using deformation theory, the motivic Adams spectral sequence, and motivic Serre finiteness (as in \cite{OstWilson}) to reduce to the characteristic zero case proved by Levine, we replace the Betti realization in Levine's proof with the stable \'etale realization. Futhermore, the techniques in this paper give us the second theorem above which is a homotopy-theoretic generalization of the \'etale version of the Suslin-Voevodsky theorem. Furthermore, we avoid using motivic Serre finiteness, and so obtain a new proof of this theorem in the special case that $k$ is algebraically closed. Though our proof uses the philosophy of Levine's proof, there are a number of technicalities that arise due to the appearance of pro-homotopy theory. Consequently, we have to work harder to prove some of the necessary lemmas and propositions. In particular, the proofs of the above theorems are not immediate consequences of our construction of the stable \'etale realization functor.\\
\\
There are a number of generalizations of the results in this paper that could be pursued. The author expects that $\textup{SH}$ fully faithfully embeds in $\textup{SH}(k)$ without the inversion of the exponential characteristic; however, the techniques of this paper do not work, at least with their current limitations. The inversion of the exponential characteristic occurs in a number of places in motivic homotopy theory. Most importantly, it occurs when dealing with resolution of singularities in the spirit of de Jong and Gabber (see Shane Kelly's work \cite{Kelly}). Furthermore, \'etale cohomology does not behave well at the residue characteristics. The author knows of no $\mathbb{A}^1$-homotopy invariant cohomology theory good enough to deal with this issue at the prime $p$. Another possible generalization that should be pursued is Galois equivariance, i.e. the dropping of the condition that $k$ is algebraically closed. There has been some progress in this direction. Now that we know the validity of motivic Serre finiteness \cite{ALP}, Heller and Ormsby have proved that, after $\eta$-completion, there is a fully faithful functor from the $C_2$-equivariant stable homotopy category to the stable motivic homotopy category of real closed fields \cite{HellerOrmsby}. The ideas behind the construction of the stable \'etale realization functor constructed in this paper may lead to Galois equivariant generalizations. Note, however, that the construction given here works well when dealing with algebraically closed fields. In order to deal with non-algebraically closed fields, we need to construct a stable \'etale realization functor that takes care of continuous Galois actions. This will be pursued in a future paper. That being said, proving that stable motivic homotopy theory subsumes some sort of Galois equivariant stable homotopy theory for fields that are not real closed requires a good theory of equivariant stable homotopy theory with profinite group actions. Perhaps, the work of Barwick \cite{Barwick} is relevant here.\\
\\
We warn the reader that the default language in this paper is that of $\infty$-categories as in \cite{HTT} and \cite{HA}. In particular, throughout the paper, $\cal{S}$ denotes the $\infty$-category of spaces. For notational simplicity, throughout this paper we view $\Sp$ as the stable $\infty$-category of $S^2$-spectra. The unit object is $\Sigma_{S^2}^{\infty}S^0$, also designated by $\mathbb{S}$.
\section{Motivic Homotopy Theory}\label{motivichomotopy}
In this section, we recall the foundations of motivic homotopy theory as well as its stable version. We will furthermore describe Voevodsky's slice tower. The bigraded slice tower can be found in \cite{Levine} in the language of triangulated categories. However, since we work in the $\infty$-categorical language, we reproduce here the basic constructions for completeness and for the benefit of the reader.\\
\\
Let $S$ be a Noetherian scheme of finite Krull dimension. Let $\Sm/S$ be the category of separated smooth schemes (of finite type) over $S$, which will henceforth be shortened to \textit{smooth schemes}. Let $\PSh_{\infty}(\Sm/S)$ and $\PSh_{\infty}^*(\Sm/S)$ be the $\infty$-categories $\textup{Fun}((\Sm/S)^{\textup{op}},\cal{S})$ and $\textup{Fun}((\Sm/S)^{\textup{op}},\cal{S}_*)$ of presheaves of spaces and pointed presheaves of spaces, respectively, on $\Sm/S$. Note that whenever $\cal{C}_0$ is an ordinary category, we identify it with its nerve $N(\cal{C}_0)$. Therefore, when we speak of $\infty$-functors from $\cal{C}_0$ to an $\infty$-category $\cal{D}$, we actually mean $\infty$-functors from the nerve $N(\cal{C}_0)$ of $\cal{C}_0$ to $\cal{D}$. Note that $N(\cal{C}_0)$ is an $\infty$-category, and that the nerve of an ordinary category coincides with the simplicial nerve of $\cal{C}_0$ viewed as a simplicial category with discrete morphisms spaces. In particular, we abuse notation to identify $\Sm/S$ with its nerve $N(\Sm/S)$. Taking $\PSh_{\infty}(\Sm/S)$ and sheafifying with respect to the Nisnevich topology and localizing with respect to $S$-projections $X\times_S\mathbb{A}_S^1\rightarrow X$ gives us the $\infty$-category $\Spc(S)$ of motivic spaces (since $S$ is assumed to be Neotherian of finite Krull dimension, this is already hypercomplete). Doing the same with $\PSh^*_{\infty}(\Sm/S)$ gives us the $\infty$-category $\Spc_*(S)$ of pointed motivic spaces. We can see from the homotopy pushout square
\[\xymatrix{\mathbb{G}_{m,S} \ar@{->}[r] \ar@{->}[d] & \mathbb{A}_S^1\simeq\ast \ar@{->}[d]\\ \ast\simeq\mathbb{A}_S^1 \ar@{->}[r] & \mathbb{P}_S^1}\]
in the $\infty$-category $\Spc_*(S)$ that the $S^1$-suspension of (pointed) $\mathbb{G}_{m,S}$ is in fact homotopy equivalent to $\mathbb{P}_S^1$ (pointed at $\infty$). Let $\Omega_{\mathbb{P}_S^1}(-):=\textup{Map}_{\Spc_*(S)}(\mathbb{P}_S^1,-)$. Let $\Spt(S)$ be the $\infty$-category of $\mathbb{P}_S^1$-spectra of motivic spaces over $S$:
\[\Spt(S):=\textup{lim}\left(\hdots\xrightarrow{\Omega_{\mathbb{P}_S^1}}\Spc_*(S)\xrightarrow{\Omega_{\mathbb{P}_S^1}}\Spc_*(S)\right),\]
where the limit is taken in the $\infty$-category of $\infty$-categories. Let $\Sigma_{\mathbb{P}_S^1}^{\infty}:\Spc_*(S)\rightarrow\Spt(S)$ be the left adjoint of $\Omega_{\mathbb{P}_S^1}^{\infty}:\Spt(S)\rightarrow\Spc_*(S)$, and call it the suspension functor. Note that such a left adjoint exists because $\Spc_*(S)$ is a presentable $\infty$-category; it is the localization of a category of presheaves. See proposition 1.4.4.4 of \cite{HA} for the existence of the left adjoint.\\
\\
We remind the reader that an $\infty$-category $\cal{C}$ is said to be \textit{accessible} if for some regular cardinal $\kappa$, there is a small $\infty$-category $\cal{C}_0$ such that $\cal{C}\simeq \textup{Ind}_{\kappa}(\cal{C}_0)$ \cite{HTT}. Basically, all this is saying is that even though $\cal{C}$ may be large, it is generated under $\kappa$-small filtered colimits by a small category.  An $\infty$-functor $F:\cal{C}\rightarrow\cal{D}$ is said to be an \textit{accessible functor} if $\cal{C}$ is an accessible $\infty$-category and there is a regular cardinal $\kappa$ such that $F$ preserves $\kappa$-small filtered colimits \cite{HTT}. Note that presentable $\infty$-categories are, by definition, accessible $\infty$-categories that admit all small colimits.\\
\\
$\mathbf{1}_S\in\Spt(S)$ denotes the motivic sphere spectrum $\Sigma_{\mathbb{P}^1}^{\infty}S_+$ over $S$. Smash product makes $\Spt(S)$ a symmetric monoidal $\infty$-category with unit object $\mathbf{1}_S$. Throughout this paper, we drop $S$ in all the above notations whenever the context makes the choice of $S$ clear.\\
\\
Throughout this paper, the motivic sphere $S^{p,q}$ will be defined as
\[(S^1)^{\wedge p}\wedge\mathbb{G}_{m,S}^{\wedge q}.\]
Note that $S^{p,q}$ usually denotes $(S^1)^{\wedge (p-q)}\wedge\mathbb{G}_{m,S}^{\wedge q}$; however, we find the above convention more convenient for our purposes.\\
\\
Given an $\infty$-category $\cal{C}$, we let $[X,Y]_{\cal{C}}:=\pi_0\textup{Map}_{\cal{C}}(X,Y)$. We will omit the subscript $\cal{C}$ whenever there is no possibility of confusion. We let $\Sigma^{p,q}:\Spt(S)\rightarrow\Spt(S)$ be the bigraded suspension functor $E\mapsto S^{p,q}\wedge E$. $\Omega^{p,q}:\Spt(S)\rightarrow\Spt(S)$ be the loop functor $E\mapsto S^{-p,-q}\wedge E$ with the expected definition. It is the right adjoint of $\Sigma^{p,q}$. If $E\in\Spt(S)$ is a motivic spectrum, then $\underline{\pi}_{p,q}(E)$ is the Nisnevich sheaf associated to the presheaf on $\Sm/S$ of abelian groups
\[X\mapsto [\Sigma^{p,q}\Sigma_{\mathbb{P}^1}^{\infty}X_+,E]\]
In the proof of the main theorem, Voevodsky's slice spectral sequence will play an essential role. We now set up the necessary notation and recall the construction of this spectral sequence. Let $\Spt(S)_{\geq(a,b)}$ be the full subcategory of $\Spt(S)$ generated under colimits and extensions by $\{\Sigma^{p,q}\Sigma_{\mathbb{P}^1}^{\infty}X_+|X\in\Sm/S, p\geq a\textup{ and }q\geq b\}$. $\Spt(S)_{\geq(a,-\infty)}$ denotes the full subcategory of $\Spt(S)$ generated under colimits and extensions by $\{\Sigma^{p,q}\Sigma^{\infty}X_+|X\in\Sm/S, p\geq a\textup{ and }q\in\mathbb{Z}\}$. Similarly, $\Spt(S)_{\geq(-\infty,b)}$ denotes the full subcategory of $\Spt(S)$ generated under colimits and extensions by $\{\Sigma^{p,q}\Sigma^{\infty}X_+|X\in\Sm/S, p\in\mathbb{Z}\textup{ and }q\geq b\}$. Note that $\Spt(S)$ is a presentable stable $\infty$-category. Indeed, the inclusion from $\cal{P}r^{R}$, the $\infty$-category of presentable $\infty$-categories with morphisms those accessible functors preserving small limits, into the $\infty$-category of $\infty$-categories preserves limits (proposition 5.5.3.18 of \cite{HTT}). Therefore, $\Spt(S)$ is a limit in $\cal{P}r^R$, which is closed under small limits, again by Proposition 5.5.3.18 of \cite{HTT}. The presentability of $\Spt(S)$ can also be seen by noting that $\Spt(S)$ is a colimit in $\cal{P}r^L\simeq(\cal{P}r^R)^{op}$, the $\infty$-category of presentable $\infty$-categories with morphisms those functors preserving small colimits. By construction, $\Spt(S)$ is a \textit{stable} $\infty$-category. As a result, $\Spt(S)_{\geq(-\infty,b)}$ is also a stable $\infty$-category. For each fixed $b\geq -\infty$, from proposition 1.4.4.11 of \cite{HA} and the stability of $\Spt(S)_{\geq(-\infty,b)}$ we deduce the existence of a $t$-structure on $\Spt(S)_{(-\infty,b)}$ with non-negative part $\Spt(S)_{\geq(0,b)}$. For each $(a,b)$ (where $a,b$ can be $-\infty$), we have a truncation functor $f_{a,b}:\Spt(S)\rightarrow\Spt(S)$ given by the truncation $\tau_{\geq (a,b)}:\Spt(S)\rightarrow\Spt(S)_{\geq (a,b)}$ followed by the inclusion $i_{a,b}:\Spt(S)_{\geq (a,b)}\rightarrow\Spt(S)$. $(i_{a,b},\tau_{\geq (a,b)})$ is an adjoint pair. For convenience, $E\mapsto E_{\geq(a,b)}$ will be the truncation $f_{a,b}E$ of $E$. Given a fixed $b$, this $t$-structure on $\Spt(S)_{\geq(-\infty,b)}$ with non-negative part $\Spt(S)_{\geq(0,b)}$ gives the cofiber sequences
\[E_{\geq(a,b)}\rightarrow E\rightarrow E_{\leq (a-1,b)}\rightarrow \Sigma^{1,0}E_{\geq (a,b)}.\]
We write $f_n:=f_{-\infty,n}$. We call $\Spt^{\textup{eff}}(S):=\Spt(S)_{\geq(-\infty,0)}$ the category of \textit{$t$-effective spectra}. For each spectrum $E\in\Spt(S)$ and each $n$, the adjunction $(i_{-\infty,n},\tau_{\geq (-\infty,n)})$ allows us to produce the natural morphism $f_nE\rightarrow E$, and the $n$-th \textit{slice} $s_nE$ is defined by the cofiber sequence
\[f_{n+1}E\rightarrow f_nE\rightarrow s_nE\rightarrow \Sigma^{1,0} f_{n+1}E.\]
Note that $f_{n+1}f_n=f_{n+1}$. For every $t$-effective spectrum $E$, these cofiber sequences fit together to produce Voevodsky's \textit{slice tower}
\[\hdots\rightarrow f_{t+1}E\rightarrow f_tE\rightarrow\hdots\rightarrow f_0E=E\]
with $n^{\textup{th}}$ layer $s_nE$. This gives us an exact couple of graded abelian groups
\[\xymatrix{\underline{\pi}_{*,0}(f_{*}E)(S)\ar@{->}[rr] && \underline{\pi}_{*,0}(f_{*}E)(S)\ar@{->}[dl]\\
&\underline{\pi}_{*,0}(s_*E)(S) \ar@{->}[ul]^{\partial},&}\]
where $\partial$ is of degree $-1$ and the unlabeled morphisms are of degree $0$. From this exact couple, we get Voevodsky's \textit{slice spectral sequence}
\[E_1^{p,q}=\underline{\pi}_{p+q}(s_qE)(S)\implies \underline{\pi}_{p+q}(E)(S),\]
where the differentials on the $r$-th page are given by
\[d_r^{p,q}:E_r^{p,q}\rightarrow E_r^{p-r-1,q+r}.\]
In section~\ref{specseq}, we recall conditions under which the convergence of this spectral sequence is guaranteed. We will also study the spectral sequence obtained via the stable $\ell$-adic \'etale realization of the above slice tower. By comparing these two spectral sequences for $S=\Spec\ k$, $k$ algebraically closed, we will deduce our comparison theorems.

\section{$\ell$-adic Homotopy Theory}\label{prohom}
The stable $\ell$-adic \'etale realization functor will take values in the $\infty$-category of $(H\mathbb{Z}/\ell)^*$-local pro-spectra (those that are \textit{cohomologically} Bousfield localized with respect to the Eilenberg-Maclane spectrum $H\mathbb{Z}/\ell$). In this section, we recall pro-categories in the spirit of Grothendieck (and Lurie in the $\infty$-categorical world) and give its universal property. This will be essential to our construction of the stable \'etale realization functor. We then discuss Bousfield localization of pro-spaces and pro-spectra using the language of $\infty$-categories. We now recall the central notion of pro-categories in the world of $\infty$-categories.   

\begin{definition}[Definition 3.1.1 of \cite{DAG13}]
Given an accessible $\infty$-category $\cal{C}$ admitting finite limits, its pro-category $\Pro(\cal{C})$ is the full subcategory of $\textup{Fun}(\cal{C},\cal{S})^{\textup{op}}$ spanned by the accessible functors $F:\cal{C}\rightarrow\cal{S}$ that preserve finite limits.
\end{definition}
The notion of a pro-category as defined above is very close to the more familiar classical notion. Indeed, if $\cal{C}$ is the nerve of an ordinary accessible category $\cal{C}_0$, then the $\infty$-category $\Pro(\cal{C})$ is the nerve of the (ordinary) pro-category $\Pro(\cal{C}_0)$ (example 3.1.3 of \cite{DAG13}).\\
\\
As an example, the $\infty$-category $\Sp$ of spectra is a presentable, hence accessible, $\infty$-category admitting finite limits. Therefore, we may construct $\Pro(\Sp)$ as above. Similarly, we have $\Pro(\cal{S})$.
\begin{remark}
\textup{If $\cal{C}$ is an accessible $\infty$-category admitting finite limits, then the collection of left-exact accessible functors $\cal{C}\rightarrow\cal{S}$ is closed under small filtered colimits in $\textup{Fun}(\cal{C},\cal{S})$. Therefore, $\Pro(\cal{C})$ is closed under small cofiltered limits. This suggests the following universal property.}
\end{remark}

\begin{proposition}[Proposition 3.1.6 of \cite{DAG13}]
Suppose $\cal{C}$ is an accessible $\infty$-category admitting finite limits, and suppose $\cal{D}$ is an $\infty$-category admitting small cofiltered limits. Let $\textup{Fun}'(\Pro(\cal{C}),\cal{D})$ be the full subcategory of $\textup{Fun}(\Pro(\cal{C}),\cal{D})$ spanned by those functors that preserve small cofiltered limits. Then the Yoneda embedding $\cal{C}\rightarrow\Pro(\cal{C})$ induces an equivalence
\[\textup{Fun}'(\Pro(\cal{C}),\cal{D})\rightarrow\textup{Fun}(\cal{C},\cal{D}).\]
\end{proposition}

\begin{remark}\label{imprem} As an application, suppose $f:\cal{C}\rightarrow\cal{D}$ is a left exact functor between accessible $\infty$-categories admitting finite limits. Composing $f$ with the Yoneda embedding $\cal{D}\rightarrow\Pro(\cal{D})$, we obtain $\cal{C}\rightarrow\Pro(\cal{D})$. By the above universal property, we obtain a functor $\Pro(f):\Pro(\cal{C})\rightarrow\Pro(\cal{D})$ that preserves small cofiltered limits. On the other hand, pre-composition with $f$ gives us a functor $\Pro(\cal{D})\rightarrow\Pro(\cal{C})$ that is the left adjoint of $\Pro(f)$. Its restriction to $\mathcal{D}$ is called the \textit{pro-left adjoint} of $f$.
\end{remark}
Given a prime $\ell$, $\cal{S}^{\ell-\textup{fc}}$ is defined as the full subcategory of $\cal{S}$ of \textit{$\ell$-finite spaces}, that is, those spaces $X$ with finitely many nontrivial homotopy groups with the additional conditions that $\pi_0X$ is finite and $\pi_iX$, $i\geq 1$, are $\ell$-finite groups. We then let $\cal{S}^{\Pro(\ell)}:=\Pro(\cal{S}^{\ell-\textup{fc}})$, and call it the $\infty$-category of \textit{$\ell$-profinite spaces}.\\
\\
There is a notion of $(H\mathbb{Z}/\ell)^*$-localization of $\Pro(\cal{S})$. A morphism $X\rightarrow Y$ in $\Pro(\cal{S})$ is called an \textit{$(H\mathbb{Z}/\ell)^*$-equivalence} if it induces an isomorphism $H^*(Y;\mathbb{Z}/\ell)\rightarrow H^*(X;\mathbb{Z}/\ell)$. A pro-space $Z$ is \textit{$(H\mathbb{Z}/\ell)^*$-local} if for every $(H\mathbb{Z}/\ell)^*$-equivalence $X\rightarrow Y$, $\textup{Map}_{\Pro(\cal{S})}(Y,Z)\rightarrow\textup{Map}_{\Pro(\cal{S})}(X,Z)$ is a weak equivalence in $\cal{S}$. $\Pro(\cal{S})^{\mathbb{Z}/\ell}$ is the full $\infty$-subcategory of $(H\mathbb{Z}/\ell)^*$-local pro-spaces in $\Pro(\cal{S})$. It is not hard to show that $\Pro(\cal{S})^{\mathbb{Z}/\ell}=\Pro(\cal{S}^{\ell-\textup{fc}})$ (see remark 3.8 of \cite{Hoyois}). We construct localization functors by using the universal property of pro-categories. The inclusion $i:\cal{S}^{\ell-fc}\hookrightarrow\cal{S}$ is accessible and left exact. Therefore, remark~\ref{imprem} implies that $i:\Pro(\cal{S})^{\mathbb{Z}/\ell}\rightarrow\Pro(\cal{S})$ admits a left adjoint $L^{\mathbb{Z}/\ell}:\Pro(\cal{S})\rightarrow\Pro(\cal{S})^{\mathbb{Z}/\ell}$.\\
\\ 
Since the $\infty$-categories $\Pro(\cal{S})^{\mathbb{Z}/\ell}$ and $\Pro(\cal{S}^{\ell-\textup{fc}})$ are equivalent, we will henceforth also call $(H\mathbb{Z}/\ell)^*$-local profinite spaces $\ell$-profinite spaces. Similarly, there is also a pointed version of this construction.\\
\\
Similarly, we may localize $\Pro(\Sp)$ with respect to the Eilenberg-Maclane spectrum $H\mathbb{Z}/\ell$ as follows. Call a morphism $E\rightarrow F$ of pro-spectra an \textit{$(H\mathbb{Z}/\ell)^*$-equivalence} if $\textup{Map}_{\Pro(\Sp)}(F,H\mathbb{Z}/\ell)\rightarrow \textup{Map}_{\Pro(\Sp)}(E,H\mathbb{Z}/\ell)$ is a weak equivalence. Note that we are abusing terminology here; when we speak of $(H\mathbb{Z}/\ell)^*$-local pro-spaces, we will mean the definition given in the previous paragraph, while if we speak of $(H\mathbb{Z}/\ell)^*$-local pro-spectra, we mean this definition. It will be clear from context which we mean. Let $\Pro(\Sp)^{H\mathbb{Z}/\ell}$ be the full subcategory of $\Pro(\Sp)$ spanned by the $(H\mathbb{Z}/\ell)^*$-local objects, that is, the full subcategory of $\Pro(\Sp)$ consisting of spectra $G$ such that for every $(H\mathbb{Z}/\ell)^*$-equivalence $E\rightarrow F$, $\textup{Map}_{\Pro(\Sp)}(F,G)\rightarrow\textup{Map}_{\Pro(\Sp)}(E,G)$ is an equivalence. Similar to the case of $\Pro(\cal{S})$, we produce an adjoint pair of functors
\[L^{H\mathbb{Z}/\ell}:\Pro(\Sp)\rightleftarrows\Pro(\Sp)^{H\mathbb{Z}/\ell}:i_{st},\]
where $i_{st}$ is the inclusion functor. The existence of the Bousfield localization functors could have also been shown by using the strict model structures of Isaksen \cite{Isaksen2}.\\
\\
As we will see in our construction of the \'etale realization functors, the categories $\Pro(\cal{S})^{\mathbb{Z}/\ell}$ and $\Pro(\Sp)^{H\mathbb{Z}/\ell}$ will play an important role for us.\\
\\
In classical stable homotopy theory, we have the adjunction between infinite suspension and infinite loop space functors
\[\Sigma_+^{\infty}:\cal{S}\rightleftarrows\Spt:\Omega^{\infty}.\]
This induces the adjunction
\[\Pro(\Sigma_+^{\infty}):\Pro(\cal{S})\rightleftarrows\Pro(\Spt):\Pro(\Omega^{\infty}),\]
from which we naturally obtain the adjunction
\[\Sigma^{\infty}_{S^2_{\ell}+}:\Pro(\cal{S})^{\mathbb{Z}/\ell}\rightleftarrows\Pro(\Spt)^{H\mathbb{Z}/\ell}:\Omega_{S^2_{\ell}}^{\infty}\]
of the Bousfield localized categories. 

\section{Shape Theory and \'Etale Realization Functors}\label{etalehomotopy}
In \cite{Levine}, Levine used the Betti realization of motivic spectra to prove his full-faithfulness result for algebraically closed fields of characteristic zero. The stable Betti realization is an easy extension of the unstable Betti realization of motivic \textit{spaces}. In positive characteristic, however, the extension of the \'etale realization of motivic spaces due to Isaksen \cite{Isaksen1} is not immediate because we need a good stable homotopy theory of pro-spaces. The main purpose of this section is to show that for $k$ algebraically closed, there is a stable \'etale realization of motivic spectra that behaves well enough for us. \\
\\
In this section, we will first describe shape theory and recall the \'etale realization of schemes. We will then describe Hoyois' (see \cite{Hoyois}) $\infty$-categorical incarnation of Isaksen's \'etale realization of motivic spaces to obtain the unstable $\ell$-adic \'etale realization fuctor
\[\Spc(S)\rightarrow\Pro(\cal{S})^{\mathbb{Z}/\ell}.\] 
\\
For our construction of the stable $\ell$-adic \'etale realization, we first construct the \'etale realization functor for Nisnevich sheaves of $E$-modules on $\Sm/S$, $E$ an $E_{\infty}$-ring spectrum. Roughly, specializing this construction to $S=\Spec\ k$, $k$ algebraically closed, and $E=\mathbb{S}$, and performing $\mathbb{A}^1$-localization and Bousfield localization, we obtain the stable $\ell$-adic \'etale realization functor of motivic spectra
\[\Spt(k)\rightarrow\Pro(\Sp)^{H\mathbb{Z}/\ell}.\]
Also taking $S=\Spec\ k$, $k$ an algebraically closed field, in Hoyois' unstable case, we will be able to compare the unstable and stable \'etale realization functors. Before doing so, here are some preliminaries on shape theory and Grothendieck sites.\\
\\
If $\tau$ is a Grothendieck topology on the category of schemes, and $X$ is a scheme, the small $\tau$-site will be the full subcategory of $\cal{S}ch/X$, the category of separated finite type $X$-schemes, generated by members of the $\tau$-coverings of $X$ and with the Gorthendieck topology induced by $\tau$. $X_{\tau}$ denotes the $\infty$-topos of sheaves of spaces on the small $\tau$-site of $X$. This assignment of $X_{\tau}$ to a scheme $X$ is functorial in the sense that a morphism of schemes $f:X\rightarrow Y$ induces a morphism of $\infty$-topoi $f_*:X_{\tau}\rightarrow Y_{\tau}$ given by $f_*(\cal{F})(-):=\cal{F}(-\times_YX)$. In this paper, the \'etale topology will be of greatest importance to us.\\
\\
Let $\Top^R$ be the $\infty$-category of $\infty$-topoi with morphisms the geometric morphisms. Recall that a geometric morphism $f_*:\cal{X}\rightarrow\cal{Y}$ between $\infty$-topoi is a functor that admits a left exact left adjoint, designated by $f^*:\cal{Y}\rightarrow\cal{X}$. Given a geometric morphism $f_*:\cal{X}\rightarrow\cal{Y}$ with left adjoint $f^*:\cal{Y}\rightarrow\cal{X}$, we have a pro-left adjoint $f_!:\cal{X}\rightarrow\Pro(\cal{Y})$ to $f^*$ given by $f_!(X)(Y):=\textup{Map}_{\cal{X}}(X,f^*Y)$. Indeed, note that by the Yoneda embedding, $f_!(X)(Y)\simeq \textup{Map}_{\Pro(\cal{S})}(f_!(X),Y)$. Therefore, we see that $f_!$ is a pro-left adjoint to $f^*$. In particular, we have a geometric morphism $\pi_*:\cal{X}\rightarrow\cal{S}$ that is unique up to homotopy ($\cal{S}$ is final in $\Top^R$ by proposition 6.3.4.1 of \cite{HTT}). This has the constant sheaf functor $\pi^*:\cal{S}\rightarrow\cal{X}$ as its left adjoint. The pro-left adjoint $\pi_!:\cal{X}\rightarrow\Pro(\cal{S})$ is given by
\[\pi_!(F)(K)=\textup{Map}_{\cal{X}}(F,\pi^*K).\]
The \textit{shape} $\Pi_{\infty}\cal{X}$ of the $\infty$-topos $\cal{X}$ is defined as the pro-space $\pi_!(\mathbf{1}_{\cal{X}})$, where $\mathbf{1}_{\cal{X}}$ is the final object of $\cal{X}$. $\Pi_{\infty}:\Top^R\rightarrow\Pro(\cal{S})$ is in fact a pro-left adjoint to the functor $\cal{S}\rightarrow\Top^R$ sending a space to the $\infty$-topos of $\infty$-sheaves on the space. See \cite{Hoyois2} for details regarding this functor $\Pi_{\infty}$.\\
\\
Important examples come from Grothendieck sites. For example, if $X$ is a scheme, then $\Pi^{\text{\'et}}_{\infty}X:=\Pi_{\infty}X_{\text{\'et}}$ is the \textit{\'etale topological type} of $X$. Note that this construction is functorial in the sense that a morphism of schemes $X\rightarrow Y$ functorially gives a morphism $\Pi^{\text{\'et}}_{\infty}X\rightarrow\Pi^{\text{\'et}}_{\infty}Y$ of \'etale topological types.
\begin{remark}
\textup{By the construction of $\Pi^{\text{\'et}}_{\infty}X$, for constant coefficients $A$, $H^*(\Pi^{\text{\'et}}_{\infty}X;A)\cong H_{\et}^*(X;A)$. This feature of the \'etale topological type allows us to prove theorems about \'etale cohomology by using algebraic topology.}
\end{remark}
The $\ell$-adic unstable \'etale realization functor on motivic spaces is constructed as follows. We first lift
\[\Pi^{\text{\'et}}_{\infty}:\Sm/S\rightarrow\Pro(\cal{S}),\]
to a functor
\[\hat{\Pi}^{\text{\'et}}_{\infty}:\Shv^{\tau}_{\infty}(\Sm/S)\rightarrow\Pro(\cal{S})\]
that preserves small homotopy colimits. Here $\tau$ is any topology coarser than the \'etale topology, for example the Nisnevich topology. 
We note that $\Pro(\cal{S})$ is cocomplete.
Consider the following diagram:
\[\xymatrix{\Sm/S\ar@{->}[r]^{\Pi^{\text{\'et}}_{\infty}} \ar@{->}[d] & \Pro(\cal{S})\\
\PSh_{\infty}(\Sm/S) \ar@{.>}[ur]_{\tilde{\Pi}^{\text{\'et}}_{\infty}} \ar@{->}[d] &\\
\Shv^{\tau}_{\infty}(\Sm/S) \ar@{.>}[uur]_{\hat{\Pi}^{\text{\'et}}_{\infty}}. &
}\]
The middle dotted arrow can be constructed uniquely from $\Pi^{\text{\'et}}_{\infty}$ because $\Pro(\cal{S})$ is cocomplete and every presheaf of spaces on $\Sm/S$ is a colimit of representable presheaves of spaces. In order to lift to all of $\Shv^{\tau}_{\infty}(\Sm/S)$, we use the following:
\begin{lemma}
If $U$ is an \'etale covering diagram in the small \'etale site of $X$, then $\Pi^{\text{\'et}}_{\infty}X$ is the colimit of the diagram of pro-spaces $\Pi^{\text{\'et}}_{\infty}U$. 
\end{lemma}
\begin{proof}
This is lemma 1.5 of \cite{Hoyois}. The proof is that the $\infty$-topos $X_{\text{\'et}}$ is the colimit in $\Top^R$ of the diagram of $\infty$-topoi $U_{\text{\'et}}$. Since $\Pi_{\infty}:\Top^R\rightarrow\Pro(\cal{S})$ is a left adjoint, it preserves colimits. The conclusion follows.
\end{proof}
Using this, we lift $\Pi^{\text{\'et}}_{\infty}$ to a left adjoint $\hat{\Pi}^{\text{\'et}}_{\infty}:\Shv^{\tau}_{\infty}(\Sm/S)\rightarrow\Pro(\cal{S})$. Alternatively, $\Pi^{\text{\'et}}_{\infty}$ can be globalized by instead working with the \textit{big} \'etale site $(\Sm/S)_{\textup{\'et}}$. Indeed, consider the global sections functor
\[\Gamma_*:\Shv^{\textup{\'et}}_{\infty}(\Sm/S)\rightarrow\cal{S}\]
with left adjoint
\[\Gamma^*:\cal{S}\rightarrow\Shv^{\textup{\'et}}_{\infty}(\Sm/S)\]
given by sending a space to the constant sheaf on the big \'etale site over $S$. This has a pro-left adjoint
\[\Gamma_!:\Shv^{\textup{\'et}}_{\infty}(\Sm/S)\rightarrow\Pro(\cal{S}).\]
Note that $\Gamma_!(X)=\Pi_{\infty}(\Shv^{\textup{\'et}}_{\infty}(\Sm/S)/X)$. If $X$ is representable, then this is the same as $\Pi_{\infty}X_{\textup{\'et}}$ (see remark 5.3 of \cite{Hoyois}).
Therefore, $\hat{\Pi}^{\text{\'et}}_{\infty}$ is given by the composition
\[\Shv^{\tau}_{\infty}(\Sm/S)\xrightarrow{a_{\textup{\'et}}}\Shv^{\textup{\'et}}_{\infty}(\Sm/S)\xrightarrow{\Gamma_!}\Pro(\cal{S}).\]
This preserves colimits as it is the composition of (restrictions of) two left adjoint functors.\\
\\
For any prime $\ell$ not equal to the residue characteristics of $S$, and $X$ a separated finite type $S$-scheme, we have that the map $H^*_{\text{\'et}}(X;\mathbb{Z}/\ell\mathbb{Z})\rightarrow H^*_{\text{\'et}}(X\times_S\mathbb{A}_S^1;\mathbb{Z}/\ell\mathbb{Z})$ induced by the projection $X\times_S\mathbb{A}_S^1\rightarrow X$ is an isomorphism (see VII, Cor. 1.2 of \cite{SGA5}). Therefore, the composition 
\[\Shv^{\tau}_{\infty}(\Sm/S) \xrightarrow{{\hat{\Pi}}^{\text{\'et}}_{\infty}} \Pro(\cal{S}) \xrightarrow{L^{\mathbb{Z}/\ell}} \Pro(\cal{S})^{\mathbb{Z}/\ell}\]
can be localized with respect to the projections $\{X\times_S\mathbb{A}^1_S\rightarrow X|X\in\Sm/S\}$ to obtain the \textit{unstable $\ell$-adic \'etale realization}
\[\textup{\'Et}_{\ell}:\Spc_{\tau}(S)\rightarrow \Pro(\cal{S})^{\mathbb{Z}/\ell}\]
of $\tau$-motivic spaces (instead of the Nisnevich topology in the construction of motivic spaces, use the topology $\tau$). In our future computations, we will need the following lemma.
\begin{lemma}
For $p,q\geq 0$, $\textup{\'Et}_{\ell}S^{p,q}\simeq K(\mathbb{Z}_{\ell},1)^{\wedge p}\wedge K(T\mu_{\ell},1)^{\wedge q}$, where $T\mu_{\ell}$ is the Tate module of $\mu_{\ell}$.
\end{lemma}
For a proof of this, see Proposition 6.2 of \cite{Hoyois}. Choose an isomorphism $T\mu_{\ell}\simeq\mathbb{Z}_{\ell}$, and identify $\textup{\'Et}_{\ell}S^{p,q}\simeq K(\mathbb{Z}_{\ell},1)^{\wedge p+q}$.\\
\\
We now construct the stable \'etale realization functor using analogues of this alternative description. Instead of considering sheaves of spaces, we can consider sheaves of $E$-modules for some $E_{\infty}$-ring spectrum $E$. In this case, consider the global sections functor
\[\Gamma^E_*:\Shv^{\textup{\'et}}_{\infty}(\Sm/S,\Mod_E)\rightarrow\Mod_E\]
with left adjoint
\[\Gamma_E^*:\Mod_E\rightarrow\Shv^{\textup{\'et}}_{\infty}(\Sm/S,\Mod_E)\]
given by sending an $E$-module $M$ to its associated constant sheaf of $E$-modules on the big \'etale site over $S$. By formal reasons, this has a pro-left adjoint
\[\Gamma^E_!:\Shv^{\textup{\'et}}_{\infty}(\Sm/S,\Mod_E)\rightarrow\Pro(\Mod_E).\]
We then consider the morphism
\[\Pi_{\infty}^{E}:\Shv^{\tau}_{\infty}(\Sm/S,\Mod_E)\xrightarrow{a_{\textup{\'et}}}\Shv^{\textup{\'et}}_{\infty}(\Sm/S,\Mod_E)\xrightarrow{\Gamma^E_!}\Pro(\Mod_E).\]
Specializing to $E=\mathbb{S}$, we obtain the commutative diagram

\[\xymatrix{\hat{\Pi}^{\textup{\'et}}_{\infty}:\Shv^{\tau}_{\infty}(\Sm/S) \ar@{->}[r]^{a_{\textup{\'et}}} \ar@{->}[d]_{\Sigma_+^{\infty}} & \Shv^{\textup{\'et}}_{\infty}(\Sm/S) \ar@{->}[r]^{\Gamma_!} \ar@{->}[d]^{\Sigma_+^{\infty}} & \Pro(\cal{S}) \ar@{->}[d]^{\Pro(\Sigma_+^{\infty})}\\
\Pi^{\mathbb{S}}_{\infty}:\Shv^{\tau}_{\infty}(\Sm/S;\Sp) \ar@{->}[r]^{a_{\textup{\'et}}} & \Shv^{\textup{\'et}}_{\infty}(\Sm/S;\Sp) \ar@{->}[r]^{\Gamma^{\mathbb{S}}_!} & \Pro(\Sp).}\]
Furthermore, we also have the commutative diagram
\[\xymatrix{\Pro(\cal{S}) \ar@{->}[d]_{\Pro(\Sigma_+^{\infty})} \ar@{->}[r]^{L^{\mathbb{Z}/\ell}} & \Pro(\cal{S})^{\mathbb{Z}/\ell} \ar@{->}[d]^{\Sigma_{S_{\ell}^1+}^{\infty}}\\ \Pro(\Sp) \ar@{->}[r]^{L^{H\mathbb{Z}/\ell}} & \Pro(\Sp)^{H\mathbb{Z}/\ell}.}\]

The previous two diagrams fit together to form
\[\xymatrix{\Shv^{\tau}_{\infty}(\Sm/S) \ar@{->}[r]^{a_{\textup{\'et}}} \ar@{->}[d]_{\Sigma_+^{\infty}} & \Shv^{\textup{\'et}}_{\infty}(\Sm/S) \ar@{->}[r]^{\Gamma_!} \ar@{->}[d]_{\Sigma_+^{\infty}} & \Pro(\cal{S}) \ar@{->}[d]^{\Pro(\Sigma_+^{\infty})} \ar@{->}[r]^{L^{\mathbb{Z}/\ell}} & \Pro(\cal{S})^{\mathbb{Z}/\ell} \ar@{->}[d]^{\Sigma_{S_{\ell}^1+}^{\infty}}\\ \Shv^{\tau}_{\infty}(\Sm/S;\Sp) \ar@{->}[r]^{a_{\textup{\'et}}} & \Shv^{\textup{\'et}}_{\infty}(\Sm/S;\Sp) \ar@{->}[r]^{\Gamma^{\mathbb{S}}_!} & \Pro(\Sp) \ar@{->}[r]^{L^{H\mathbb{Z}/\ell}} & \Pro(\Sp)^{H\mathbb{Z}/\ell}.}\]
 
As in the construction of the unstable \'etale realization, we can see that $L^{H\mathbb{Z}/\ell}\Pi^{\textup{E}}_{\infty}$ factors through the $\mathbb{A}^1$-localized $\tau$-sheaves of $E$-modules. This gives us the \textit{$S^1$-stable} \'etale realization
\[\textup{\'Et}^{\Sp}_{\ell}:\Spt_{\tau}^{S^1}(S)\rightarrow\Pro(\Sp)^{H\mathbb{Z}/\ell}.\]
These give us the following commutative diagram of \'etale realizations: 
\[\xymatrix{\textup{\'Et}_{\ell}:\Spc_{\tau}(S) \ar@{->}[r]^{a_\textup{\'et}} \ar@{->}[d]_{\Sigma_+^{\infty}} & \Spc_{\textup{\'et}}(S) \ar@{->}[r]^{L^{\mathbb{Z}/\ell}\Gamma_!} \ar@{->}[d]_{\Sigma_+^{\infty}} & \Pro(\cal{S})^{\mathbb{Z}/\ell} \ar@{->}[d]^{\Sigma_{S_{\ell}^1+}^{\infty}}\\ \textup{\'Et}^{\Sp}_{\ell}:\Spt_{\tau}^{S^1}(S) \ar@{->}[r]^{a_\textup{\'et}} & \Spt_{\textup{\'et}}^{S^1}(S) \ar@{->}[r]^{L^{H\mathbb{Z}/\ell}\Gamma^{\mathbb{S}}_!} & \Pro(\Sp)^{H\mathbb{Z}/\ell}}\]
From now on, we restrict to $S=\Spec\ k$, where $k$ is a separably closed field. The final step in our construction is to invert $\mathbb{G}_m$ in the bottom row of the above diagram specialized to $S=\Spec\ k$. We do so by showing that there is a commutative diagram of the following form.
\[\xymatrix{\textup{\'Et}_{\ell}:\Spc_{\tau}(k) \ar@{->}[r]^{a_\textup{\'et}} \ar@{->}[d]_{\Sigma_+^{\infty}} & \Spc_{\textup{\'et}}(k) \ar@{->}[r]^{L^{\mathbb{Z}/\ell}\Gamma_!} \ar@{->}[d]_{\Sigma_+^{\infty}} & \Pro(\cal{S})^{\mathbb{Z}/\ell} \ar@{->}[d]^{\Sigma_{S_{\ell}^1+}^{\infty}}\\ \textup{\'Et}^{\Sp}_{\ell}:\Spt_{\tau}^{S^1}(k) \ar@{->}[r]^{a_\textup{\'et}} \ar@{->}[d]_{\Sigma^{\infty}_{\mathbb{G}_m}} & \Spt_{\textup{\'et}}^{S^1}(k) \ar@{->}[r]^{L^{H\mathbb{Z}/\ell}\Gamma^{\mathbb{S}}_!} \ar@{->}[d]_{\Sigma^{\infty}_{\mathbb{G}_m}} & \Pro(\Sp)^{H\mathbb{Z}/\ell} \ar@{->}[d]^{\Sigma_{S_{\ell}^1}^{\infty}}\\ \Spt_{\tau}(k) \ar@{->}[r]^{a_\textup{\'et}} & \Spt_{\textup{\'et}}(k) \ar@{.>}[r] & \Pro(\Sp)^{H\mathbb{Z}/\ell}}\]
We need to construct the dotted functor
\[\Spt_{\textup{\'et}}(k) \rightarrow \Pro(\Sp)^{H\mathbb{Z}/\ell}\]
so that the lower right square commutes. We do so by first checking that the square
\[\xymatrix{\Spt^{S^1}_{\textup{\'et}}(k)^{\omega} \ar@{->}[r]^{\Sigma_{\mathbb{G}_m}} \ar@{->}[d]_{F} & \Spt^{S^1}_{\textup{\'et}}(k)^{\omega} \ar@{->}[d]_{F} \\ \Pro(\Sp)^{H\mathbb{Z}/\ell} \ar@{->}[r]^{\Sigma_{S_{\ell}^1}} & \Pro(\Sp)^{H\mathbb{Z}/\ell},}\]
where $F:=L^{H\mathbb{Z}/\ell}\Gamma^{\mathbb{S}}_!$ and $\Spt^{S^1}_{\textup{\'et}}(k)^{\omega}$ is the full subcategory of $\Spt^{S^1}_{\textup{\'et}}(k)$ spanned by the $\omega$-compact objects, commutes. Indeed, take $\Sigma_{S^1}^{a+\infty}X_+$, $X$ a smooth $k$-scheme and $a\in\mathbb{Z}$. Such objects compactly generate $\Spt^{S^1}_{\textup{\'et}}(k)$ under colimits and extensions ($k$ is algebraically closed, and so a descent spectral sequence shows that $\Spt^{S^1}_{\textup{\'et}}(k)$ is compactly generated), and our functors are exact and preserve colimits. It therefore suffices to check commutativity of the above square on such objects. Note that if $\Spc_{\text{fin}}(k)$ is the full subcategory of $\Spc(k)$ generated by $\Sm/k$ under finite colimits, then $\text{\'Et}_{\ell}|_{\Spc_{\text{fin}}(k)}:\Spc_{\text{fin}}(\Sm/k)\rightarrow\Pro(\cal{S})^{\mathbb{Z}/\ell}$ preserves finite products. This is lemma 6.1 of \cite{Hoyois}. We remark that it is not necessarily true that $\text{\'Et}_{\ell}$ preserves finite products on all of $\Spc(k)$ because the Kunneth formula requires some finiteness conditions. Using this, we deduce that 
\begin{eqnarray*}
F(\Sigma_{\mathbb{G}_m}\Sigma_{S^1}^{a+\infty}X_+) &\simeq& L^{H\mathbb{Z}/\ell}\Gamma^{\mathbb{S}}_!(\Sigma_{S^1}^{\infty}(\Sigma_{\mathbb{G}_m}\Sigma_{S^1}^aX_+))\\ &\simeq& \Sigma_{S_{\ell}^1}^{\infty}L^{\mathbb{Z}/\ell}\Gamma_!(\Sigma_{\mathbb{G}_m}\Sigma_{S^1}^aX_+)\\ &\simeq& \Sigma_{S_{\ell}^1}^{a+1+\infty}L^{\mathbb{Z}/\ell}\Gamma_!(X_+)\\ &\simeq& \Sigma_{S_{\ell}^1}L^{H\mathbb{Z}/\ell}\Gamma^{\mathbb{S}}_!(\Sigma_{S^1}^{a+\infty}X_+)\\ &\simeq& \Sigma_{S_{\ell}^1}F(\Sigma_{S^1}^{a+\infty}X_+).
\end{eqnarray*}
By proposition 5.5.7.8 of \cite{HTT}, $\Spt_{\et}(k)^{\omega}$ is the colimit
\[\textup{colim}\left(\Spt^{S^1}_{\et}(k)^{\omega}\xrightarrow{\Sigma_{\mathbb{G}_m}}\Spt^{S^1}_{\et}(k)^{\omega}\xrightarrow{\Sigma_{\mathbb{G}_m}}\hdots\right)\]
in $\cal{C}at_{\infty}(\omega)$, the $\infty$-category of small $\infty$-categories admitting $\omega$-filtered colimits and with morphisms those functors that preserve $\omega$-filtered colimits. We can similarly define $\widehat{\cal{C}at}_{\infty}(\omega)$, but with big $\infty$-categories instead. $\cal{C}at_{\infty}(\omega)\subset\widehat{\cal{C}at}_{\infty}(\omega)$ preserves colimits. Taking colimits of the above square in the category $\widehat{\cal{C}at}_{\infty}(\omega)$, we obtain the commutative square
\[\xymatrix{\Spt^{S^1}_{\textup{\'et}}(k)^{\omega} \ar@{->}[r]^{\Sigma^{\infty}_{\mathbb{G}_m}} \ar@{->}[d]_{F} & \Spt_{\textup{\'et}}(k)^{\omega} \ar@{->}[d]_{F} \\ \Pro(\Sp)^{H\mathbb{Z}/\ell} \ar@{->}[r]^{\Sigma^{\infty}_{S_{\ell}^1}} & \Pro(\Sp)^{H\mathbb{Z}/\ell},}\]
from which we get the commutative square
\[\xymatrix{\Spt^{S^1}_{\textup{\'et}}(k) \ar@{->}[r]^{\Sigma^{\infty}_{\mathbb{G}_m}} \ar@{->}[d]_{F} & \Spt_{\textup{\'et}}(k) \ar@{->}[d]_{F} \\ \Pro(\Sp)^{H\mathbb{Z}/\ell} \ar@{->}[r]^{\Sigma^{\infty}_{S_{\ell}^1}} & \Pro(\Sp)^{H\mathbb{Z}/\ell},}\]
by left Kan extensions. We thus obtain the commutative diagram
\[\xymatrix{\textup{\'Et}_{\ell}:\Spc_{\tau}(k) \ar@{->}[r]^{a_\textup{\'et}} \ar@{->}[d]_{\Sigma_{\mathbb{P}^1+}^{\infty}} & \Spc_{\textup{\'et}}(k) \ar@{->}[r]^{L^{\mathbb{Z}/\ell}\Gamma_!} \ar@{->}[d]_{\Sigma_{\mathbb{P}^1+}^{\infty}} & \Pro(\cal{S})^{\mathbb{Z}/\ell} \ar@{->}[d]^{\Sigma_{S_{\ell}^2+}^{\infty}}\\ \underline{\textup{\'Et}}_{\ell}:\Spt_{\tau}(k) \ar@{->}[r]^{a_\textup{\'et}} & \Spt_{\textup{\'et}}(k) \ar@{->}[r]^{L^{H\mathbb{Z}/\ell}\Gamma^{\mathbb{S}}_!} & \Pro(\Sp)^{H\mathbb{Z}/\ell}.}\]
The lower row corresponds to an exact functor $\underline{\textup{\'Et}}_{\ell}$, which we call the \textit{stable} $\ell$-adic \'etale realization functor. Furthermore, it is not hard to see that this functor is oplax monoidal as it is constructed from the pro-left adjoint of a monoidal functor.

\section{Two Spectral Sequences}\label{specseq}
In this section, we discuss the convergence of Voevodsky's slice spectral sequence and that of the spectral sequence obtained from the stable $\ell$-adic \'etale realization of the slice tower. We show that both spectral sequences are strongly convergent under suitable conditions. It is by comparing these two spectral sequences that we will prove our main theorems.\\
\\
Let $\Spt_{\textup{fin}}(k)$ to be the smallest full subcategory of $\Spt(k)$ containing $\Sigma_{\mathbb{P}^1}^n\Sigma^{\infty}_+X$, $n\in\mathbb{Z}$, $X\in\Sm/k$, and closed under finite colimits and extensions. Call such motivic spectra \textit{finite motivic spectra}. Given a finite effective motivic spectrum $E\in\Spt_{\textup{fin}}^{\textup{eff}}(k)$, we obtain the slice tower
\[\hdots\rightarrow f_{t+1}(E)\rightarrow f_t(E)\rightarrow\hdots\rightarrow f_1(E)\rightarrow f_0(E)=E\]
as discussed in section~\ref{motivichomotopy}. This gives rise to the following unrolled exact couple:
\begingroup
    \fontsize{9.5pt}{11pt}\selectfont
$$%\xymatrixcolsep{5pc}
\xymatrix{\hdots\underline{\pi}_{*,0}(f_{q+1}(E))(k)\ar@{->}[r]^{i}& \underline{\pi}_{*,0}(f_{q}(E))(k)\ar@{->}[r] \ar@{->}[d]_{j}&\hdots\ar@{->}[r] &\underline{\pi}_{*,0}(f_{1}(E))(k)\ar@{->}[r]^{i}& \underline{\pi}_{*,0}(E)(k) \ar@{->}[d]_{j}\\
\hdots & \underline{\pi}_{*,0}(s_{q}(E))(k)\ar@{->}[ul]^{k} & & \hdots & \underline{\pi}_{*,0}(s_{1}(E))(k). \ar@{->}[ul]^{k}}
$$
\endgroup
Since $\mathbb{Z}_{\ell}$ is a flat $\mathbb{Z}$-module, tensoring with $\mathbb{Z}_{\ell}$ gives the exact couple
\begingroup
    \fontsize{9.5pt}{11pt}\selectfont
%\tiny{
$$%\xymatrixcolsep{pc}
\xymatrix{\hdots\underline{\pi}_{*,0}(f_{q+1}(E))(k)\otimes_{\mathbb{Z}}\mathbb{Z}_{\ell}\ar@{->}[r]^{i\otimes_{\mathbb{Z}}\mathbb{Z}_{\ell}}& \underline{\pi}_{*,0}(f_{q}(E))(k)\otimes_{\mathbb{Z}}\mathbb{Z}_{\ell}\ar@{->}[r] \ar@{->}[d]_{j\otimes_{\mathbb{Z}}\mathbb{Z}_{\ell}}&\hdots\ar@{->}[r] &\underline{\pi}_{*,0}(f_1(E))(k)\otimes_{\mathbb{Z}}\mathbb{Z}_{\ell}\ar@{->}[r]^{i\otimes_{\mathbb{Z}}\mathbb{Z}_{\ell}}& \underline{\pi}_{*,0}(E)(k)\otimes_{\mathbb{Z}}\mathbb{Z}_{\ell} \ar@{->}[d]_{j\otimes_{\mathbb{Z}}\mathbb{Z}_{\ell}}\\
\hdots & \underline{\pi}_{*,0}(s_{q}(E))(k)\otimes_{\mathbb{Z}}\mathbb{Z}_{\ell}\ar@{->}[ul]^{k\otimes_{\mathbb{Z}}\mathbb{Z}_{\ell}} & & \hdots & \underline{\pi}_{*,0}(s_{0}(E))(k)\otimes_{\mathbb{Z}}\mathbb{Z}_{\ell}. \ar@{->}[ul]^{k\otimes_{\mathbb{Z}}\mathbb{Z}_{\ell}}}
$$
%}
\endgroup
For each $q$, consider the filtration 
$$F^{n,q}:=\text{Filt}^n\underline{\pi}_{*,0}(f_q(E))(k)[1/p]:=\text{Im}(i^{(n)}:\underline{\pi}_{*,0}(f_{n+q}(E))(k)\rightarrow \underline{\pi}_{*,0}(f_q(E))(k))[1/p].$$ 
By theorem 6.3 of \cite{slice}, since $E$ finite, we have that $F^{n,q}=0$ for $n\gg 0$, where this lower bound is not universal and depends on $*$. Note that $F^{n,q}\otimes_{\mathbb{Z}}\mathbb{Z}_{\ell}=\text{Im}(i^{(n)})\otimes_{\mathbb{Z}}\mathbb{Z}_{\ell}\cong \text{Im}(i^{(n)}\otimes_{\mathbb{Z}}\mathbb{Z}_{\ell})$. The claim is that this last exact couple induces a \textit{strongly convergent} spectral sequence
\[E_1^{p,q}=\underline{\pi}_{p+q,0}(s_qE)(k)\otimes_{\mathbb{Z}}\mathbb{Z}_{\ell}\implies \underline{\pi}_{p+q,0}(E)(k)\otimes_{\mathbb{Z}}\mathbb{Z}_{\ell}\]
where the differentials on the $r$-th page are given by
\[d_r^{p,q}:E_r^{p,q}\rightarrow E_r^{p-r-1,q+r}.\]
Let $Q^q:=\bigcap_n(F^{n,q}\otimes_{\mathbb{Z}}\mathbb{Z}_{\ell})$ and $RQ^s:=R\lim_nF^{n,q}\otimes_{\mathbb{Z}}\mathbb{Z}_{\ell}$. In our case, $Q^q=0$ for every $q$ by what was said above. It is also easy to see that since $F^{n,q}\otimes_{\mathbb{Z}}\mathbb{Z}_{\ell}=0$ for $n\gg 0$, $RQ^s=0$. Note that $\lim_q\underline{\pi}_{*,0}(f_q(E))(k)\otimes_{\mathbb{Z}}\mathbb{Z}_{\ell}=\lim_qQ^q=0$, and so the spectral sequence is \textit{Hausdorff}. By the Mittag-Lefler condition,
\[\xymatrix{0\ar@{->}[r] & R\lim_q Q^q \ar@{->}[r] & R\lim_q(\underline{\pi}_{*,0}(f_q(E))(k)\otimes_{\mathbb{Z}}\mathbb{Z}_{\ell}) \ar@{->}[r] & \lim_q RQ^q \ar@{->}[r] & 0,}\]
and these conditions, we deduce that $R\lim_q(\underline{\pi}_{*,0}(f_q(E))(k)\otimes_{\mathbb{Z}}\mathbb{Z}_{\ell})=0$, i.e. the associated spectral sequence is \textit{complete}. We conclude that the spectral sequence
\[E_1^{p,q}=\underline{\pi}_{p+q,0}(s_q(E))(k)\otimes_{\mathbb{Z}}\mathbb{Z}_{\ell}\implies \underline{\pi}_{p+q,0}(E)(k)\otimes_{\mathbb{Z}}\mathbb{Z}_{\ell}\]
is \textit{at least conditionally convergent}. Consider the exact sequence
\[\xymatrix{0 \ar@{->}[r] & \frac{F^{n,0}\otimes_{\mathbb{Z}}\mathbb{Z}_{\ell}}{F^{n+1,0}\otimes_{\mathbb{Z}}\mathbb{Z}_{\ell}} \ar@{->}[r] & E^n_{\infty} \ar@{->}[r] & Q^{n+1} \ar@{->}[r] & Q^n \ar@{->}[r] & RE^n_{\infty} \ar@{->}[r] & RQ^{n+1} \ar@{->}[r] & RQ^n \ar@{->}[r] & 0}\]
of lemma 5.6 of \cite{Boardman}. Since $Q^n=0$ for all $n$, we conclude from the exact sequence that $RE^n_{\infty}=0$. As a result of conditional convergence, $RE^n_{\infty}=0$, and the fact that $E^r_s=0$ for $s<0$, we can conclude via Theorem 7.1 of \cite{Boardman} that for every $E\in\Spt_{\textup{fin}}^{\textup{eff}}(k)$, we have the \textit{strongly convergent} spectral sequence
\[E_1^{p,q}=\underline{\pi}_{p+q,0}(s_q(E))(k)\otimes_{\mathbb{Z}}\mathbb{Z}_{\ell}\implies \underline{\pi}_{p+q,0}(E)(k)\otimes_{\mathbb{Z}}\mathbb{Z}_{\ell},\]
where the differentials on the $r$-th page are given by
\[d_r^{p,q}:E_r^{p,q}\rightarrow E_r^{p-r-1,q+r}.\]
We will now study the stable $\ell$-adic \'etale analogue of the slice tower above.
Applying the stable $\ell$-adic \'etale realization $\underline{\textup{\'Et}}_{\ell}$, we obtain the tower
\[\hdots\rightarrow \underline{\textup{\'Et}}_{\ell}\left(f_{t+1}(E)\right)\rightarrow \underline{\textup{\'Et}}_{\ell}\left(f_{t}(E)\right)\rightarrow\hdots\rightarrow \underline{\textup{\'Et}}_{\ell}\left(E\right).\]
Since $\underline{\textup{\'Et}}_{\ell}$ is an exact functor, we obtain a spectral sequence
\[E_1^{p,q}=\pi_{p+q}\left(\underline{\textup{\'Et}}_{\ell}\left(s_q(E)\right)\right)\implies \pi_{p+q}\left(\underline{\textup{\'Et}}_{\ell}\left(E\right)\right),\]
where the differentials on the $r$-th page are given by
\[d_r^{p,q}:E_r^{p,q}\rightarrow E_r^{p-r-1,q+r}.\]
In order to establish the strong convergence of this spectral sequence, we need the following connectivity result. Recall that a $\mathbb{P}^1$-spectrum $E\in\Spt(k)$ is \textit{topologically $N$-connected} if for all $a\leq N$ and $b\in\mathbb{Z}$, $\underline{\pi}_{a,b}(E)=0$.
\begin{proposition}\label{etaleconnectivity}
If $E\in\Spt(k)$ is topologically $(N-1)$-connected, then $\underline{\textup{\'Et}}_{\ell}(f_q(E))$ is $(q+N-1)$-connected. 
\end{proposition} 

\begin{proof}
Note that $\Sigma^{m,n}\Sigma^{\infty}_+X$ has stable $\ell$-adic \'etale realization $\Sigma^{m+n}_{S_{\ell}^1}\Sigma^{\infty}_{S_{\ell}^2}\textup{\'Et}_{\ell}X_+$ which is $(m+n-1)$-connected. $\Spt(k)_{\geq (N,q)}$ is generated under homotopy colimits and extensions by $\{\Sigma^{m,n}\Sigma^{\infty}_+Y|Y\in\Sm/k,m\geq N\textup{ and }n\geq q\}$, and so every $E\in\Spt(k)_{\geq (N,q)}$ has \'etale realization that is $(N+q-1)$-connected. Therefore, by proposition 4.7 of \cite{Levine}, if $E$ is topologically $(N-1)$-connected, then $\underline{\textup{\'Et}}_{\ell}(f_q(E))$ is $(q+N-1)$-connected, as required.
\end{proof}

By proposition~\ref{etaleconnectivity}, $\underline{\textup{\'Et}}_{\ell}\left(f_q(E)\right)$ is $(q-1)$-connected in $\Pro(\Sp)^{H\mathbb{Z}/\ell}$, and so this spectral sequence is also strongly convergent. As a result, we conclude that
\begin{proposition}
If $E\in\Spt_{\textup{fin}}^{\textup{eff}}(k)$ is a finite effective motivic spectrum over an algebraically closed field $k$, the stable $\ell$-adic \'etale realization of the slice tower gives a strongly convergent spectral sequence
\[E_1^{p,q}=\pi_{p+q}\left(\underline{\textup{\'Et}}_{\ell}\left(s_q(E)\right)\right)\implies \pi_{p+q}\left(\underline{\textup{\'Et}}_{\ell}\left(E\right)\right),\]
where the differentials on the $r$-th page are given by
\[d_r^{p,q}:E_r^{p,q}\rightarrow E_r^{p-r-1,q+r}.\]
\end{proposition}

\section{Proof of the Main Theorem}\label{maintheorem}
In this section, we prove our main result saying that up to inversion of the exponential characteristic $p$, classical stable homotopy theory is subsumed by stable motivic homotopy theory.\\
\\
Let $\Sigma_{\mathbb{G}_m}:\cal{S}_*\rightarrow\Spc_*(k)$ be the functor that suspends pointed spaces by $\mathbb{G}_m$. We can stabilize this functor to obtain the functor $c:=\Sigma^{\infty}_{\mathbb{G}_m}:\Sp\rightarrow\Spt(k)$. The main theorem of this paper is that away from the exponential characteristic of $k$, $k$ algebraically closed, $c$ is fully faithful.
\begin{theorem}\label{full}
Let $k$ be an algebraically closed field of exponential characteristic $p$. Then
\[c[1/p]:\textup{SH}[1/p]\rightarrow\textup{SH}(k)[1/p]\]
is a fully faithful functor.
\end{theorem}
We deduce this theorem as a corollary of the following special case.
\begin{theorem}\label{main}
Let $k$ be an algebraically closed field of exponential characteristic $p$. Then for every prime $\ell\neq p$, the stable $\ell$-adic \'etale realization induces an isomorphism
\[\underline{\textup{\'Et}}_{\ell\ast}:\underline{\pi}_{*,0}(\mathbf{1}_k)(k)\otimes_{\mathbb{Z}}\mathbb{Z}_{\ell}\rightarrow \pi_*\left(\mathbb{S}\right)\otimes_{\mathbb{Z}}\mathbb{Z}_{\ell}.\]
\end{theorem}
Indeed, assuming thereom~\ref{main}, we prove theorem~\ref{full} as follows.
\begin{proof}[Proof of theorem~\ref{full}.] Note that for any space $X\in\cal{S}$,
\[\underline{\textup{\'Et}}_{\ell}\circ c(\Sigma_{S^1}^{\infty}X_+)=\underline{\textup{\'Et}}_{\ell}(\Sigma^{\infty}_{\mathbb{P}_1}X_+)\simeq \Sigma^{\infty}_{S_{\ell}^2}\textup{\'Et}_{\ell}X_+\]
the latter spectrum being in $\Pro(\Sp)^{H\mathbb{Z}/\ell}$. For each $\ell\neq p$ a prime, the composition
\begingroup
    \fontsize{9.5pt}{11pt}\selectfont
		
\[\left[\Sigma^n_{S^1}\mathbb{S},\mathbb{S}\right]_{\Sp}\otimes\mathbb{Z}_{\ell}\xrightarrow{c_*\otimes\mathbb{Z}_{\ell}} \left[\Sigma^n_{S^1}\Sigma^{\infty}_{\mathbb{G}_m}\mathbb{S},\Sigma^{\infty}_{\mathbb{G}_m}\mathbb{S}\right]_{\Spt(k)}\otimes\mathbb{Z}_{\ell}=\left[\Sigma^{n}_{S^1}\mathbf{1}_k,\mathbf{1}_k\right]_{\Spt(k)}\otimes\mathbb{Z}_{\ell}\xrightarrow{\underline{\textup{\'Et}}_{\ell\ast}} \left[\Sigma^{n}_{S_{\ell}^1}L^{H\mathbb{Z}/\ell}\mathbb{S},L^{H\mathbb{Z}/\ell}\mathbb{S}\right]_{\Pro(\Sp)^{H\mathbb{Z}/\ell}}\]

\endgroup

is an isomorphism. Since theorem~\ref{main} implies that the second map is an isomorphism as well, we deduce that $c_*\otimes\mathbb{Z}_{\ell}$ is an isomorphism. Since $\mathbb{Z}_{(\ell)}$ is a local Noetherian ring, its completion $\mathbb{Z}_{\ell}$ is a faithfully flat $\mathbb{Z}_{(\ell)}$-module. Therefore, $c_*\otimes\mathbb{Z}_{(\ell)}$ is an isomorphism for each $\ell\neq p$, which, in turn, implies that $c[1/p]_*$ is an isomorphism. Let $\cal{R}$ be the full subcategory of $\Sp$ having all objects $E\in\Sp$ such that
\[c[1/p]_*:[\Sigma^n_{S^1}\mathbb{S},E][1/p]\rightarrow [c(\Sigma^n_{S^1}\mathbb{S}),c(E)][1/p]\]
is an isomorphism for every $n$. $c[1/p]_*$ preserves small colimits and both $\Sigma^n_{S^1}\mathbb{S}$ and $c(\Sigma^n_{S^1}\mathbb{S})\simeq \Sigma^n_{S^1}\mathbf{1}_k$ are compact. By the above calculations, $\mathbb{S}\in\cal{R}$. Since $\Sp$ is generated by $\mathbb{S}$, we conclude that $\cal{R}=\Sp$, that is 
\[c[1/p]_*:[\Sigma^n_{S^1}\mathbb{S},E][1/p]\rightarrow [c(\Sigma^n_{S^1}\mathbb{S}),c(E)][1/p]\]
is an isomorphism for every $E\in\Sp$. We can define $\cal{L}$ as the full subcategory of $\Sp$ of spectra $E$ such that
\[c[1/p]_*:[E,F][1/p]\rightarrow [c(E),c(F)][1/p]\]
is an isomorphism for every $F\in\Sp$. A similar density argument as above shows that $\cal{L}=\Sp$. We conclude that $c[1/p]$ is a fully faithful functor.
\end{proof}

A direct corollary of theorem~\ref{main} is motivic Serre finiteness for algebraically closed fields.
\begin{corollary}[Special Case of Motivic Serre Finiteness]\label{msf}
For $k$ an algebraically closed field and $n>0$ an integer, $\underline{\pi}_{n,0}(\mathbf{1}_k)(k)\otimes\mathbb{Q}=0$. 
\end{corollary}

\begin{remark}\label{msfr}
\textup{It is now a theorem of Ananyevskiy, Levine, and Panin that motivic Serre finiteness is valid for all fields $k$ \cite{ALP}.}
\end{remark}

In order to prove theorem~\ref{main}, we compare the two spectral sequences of the previous section to prove a number of comparison theorems. In combination, these comparison theorems will give us the above theorem. Given a triangulated category $\cal{T}$, we define $\cal{T}_{\textup{tor}}$ to be the full subcategory with objects $E$ such that $\textup{Hom}_{\cal{T}}(A,E)\otimes\mathbb{Q}=0$ for every compact object $A$ in $\cal{T}$. These objects will be called \textit{torsion} objects. Given a stable $\infty$-category $\cal{C}$, $\cal{C}_{\textup{tor}}$ is the full subcategory consisting of the torsion objects in the triangulated category $\textup{h}\cal{C}$. We need the following result.
\begin{lemma}(Theorem 5.8 of \cite{HKO})
Let $k$ be a perfect field of exponential characteristic $p$, and let $R$ be a commutative ring in which $p$ is invertible. There is a symmetric monoidal Quillen equivlance
\[\Phi:\Mod_{MR}\rightleftarrows \textup{DM}(k;R):\Psi.\]
\end{lemma}

Taking $R=\mathbb{Z}[1/p]$, we obtain an equivalence $\textup{DM}\left(k;\mathbb{Z}[1/p]\right)\simeq\Mod_{M\mathbb{Z}[1/p]}$. Let 
\[\textup{EM}:\textup{DM}\left(k;\mathbb{Z}[1/p]\right)\rightarrow \Spt(k)\]
be the functor given by the equivalence $\textup{DM}\left(k;\mathbb{Z}[1/p]\right)\simeq\Mod_{M\mathbb{Z}[1/p]}$ followed by the forgetful functor $\Mod_{M\mathbb{Z}[1/p]}\rightarrow\Spt(k)$. Furthermore, let 
\[\mathbb{Z}[1/p]^{\textup{tr}}:\Spt(k)\rightarrow\textup{DM}\left(k;\mathbb{Z}[1/p]\right)\]
be the left adjoint of $\textup{EM}$. These functors have the property that $(EM\circ \mathbb{Z}[1/p]^{\textup{tr}})(X)=M\mathbb{Z}[1/p]\wedge X_+$ for every $X\in\Sm/k$. We claim that $\textup{EM}$ preserves small (homotopy) colimits. Indeed, it is easy to show that an $\infty$-functor $R:\cal{A}\rightarrow\cal{B}$ between cocomplete compactly generated stable $\infty$-categories preserves all small colimits if it has a left adjoint $L:\cal{B}\rightarrow\cal{A}$ with the property that $L(B)$ is compact for $B$ in a set of compact generators of $\cal{B}$. Applying this to $R=\textup{EM}$ and $L=\mathbb{Z}[1/p]^{\textup{tr}}$, choosing the compact generators $\Sigma^{p,q}\Sigma_{\mathbb{P}^1}^{\infty}X_+$, $p,q\in\mathbb{Z}$, $X\in\Sm/k$, and noting that $\mathbb{Z}[1/p]^{\textup{tr}}(\Sigma^{p,q}\Sigma_{\mathbb{P}^1}^{\infty}X_+)\simeq\mathbb{Z}[1/p]^{\textup{tr}}(X)(q)[p+q]$ is compact in $\textup{DM}\left(k;\mathbb{Z}[1/p]\right)$, we deduce that $\textup{EM}$ preserves all small (homotopy) colimits.\\
\\
By theorem 3.6.22 of \cite{Palaez}, for each motivic spectrum $E$, $s_qE$ has the natural structure of an $M\mathbb{Z}$-module, and so $s_q(E)[1/p]$ is an $M\mathbb{Z}[1/p]$-module. By the above discussion, $s_q(E)[1/p]=EM(\pi_q^{\mu}(E)(q)[2q])$ for some motive $\pi_q^{\mu}(E)\in\textup{DM}^{\textup{eff}}(k;\mathbb{Z}[1/p])$.  
\begin{lemma}(Lemmas 6.1 and 6.2 of \cite{Levine})\label{2}
Suppose $k$ is a field of finite cohomological dimension. Then for $X\in\Sm/k$ of dimension $d$ over $k$, and for $q\geq d+1$, $f_q(\Sigma^{\infty}_{\mathbb{P}^1}X_+)$ and $s_q(\Sigma^{\infty}_{\mathbb{P}^1}X_+)$ are in $\Spt(k)_{\textup{tor}}$, and $\pi_q^{\mu}(\Sigma^{\infty}_{\mathbb{P}^1}X_+)$ is in $\textup{DM}^{\textup{eff}}(k)_{\textup{tor}}$. In particular, for $q\geq 1$, $f_q(\mathbf{1})$ and $s_q(\mathbf{1})$ are in $\Spt(k)_{\textup{tor}}$, and $\pi_q^{\mu}(\mathbf{1})$ is in $\textup{DM}^{\textup{eff}}(k)_{\textup{tor}}$.
\end{lemma}

\begin{lemma}(Lemma 6.3 of \cite{Levine})\label{torlem}
If $E\in\Spt(k)_{\textup{tor}}$, then for all $q$, $f_q(E),s_q(E)\in\Spt(k)_{\textup{tor}}$ and $\pi_q^{\mu}(E)\in\textup{DM}^{\textup{eff}}(k)_{\textup{tor}}$.
\end{lemma}

In order to prove the next crucial proposition, we need the following lemma on motivic cohomology. Though this is well-known, we provide its short proof for the convenience of the reader.

\begin{lemma}\label{coho}
Suppose $\tau\in\{\et,\textup{Nis}\}$ is a choice of a Grothendieck topology, $R$ a commutative ring, and $k$ any field of finite cohomological dimension. Designate by $D_{\tau}:=D(\textup{Shv}^{tr}_{\tau}(k;R))$ the unbounded derived category associated to the abelian category of $\tau$-sheaves of $R$-modules with transfers $\textup{Shv}^{tr}_{\tau}(k;R)$. Then for every smooth $k$-scheme $X$, every complex $\underline{K}\in D_{\tau}$, and every $i\in\mathbb{Z}$, we have
\[\Ext^i_{\textup{Shv}^{tr}_{\tau}(k;R)}(R_{\textup{tr}}(X),\underline{K})=\mathbb{H}^i_{\tau}(X;\underline{K}).\]
\end{lemma}

\begin{proof}
First, we reproduce the proof of this when the complex $\underline{K}$ is a single $\tau$-sheaf of $R$-modules with transfers (concentrated in degree $0$), i.e. we show that if $F$ is a $\tau$-sheaf of $R$-modules with tranfer, then for each smooth $k$-scheme $X$ and each $i\in\mathbb{Z}$ we have
\[\Ext^i_{\textup{Shv}^{tr}_{\tau}(k;R)}(R_{\textup{tr}}(X),F)=\textup{H}^i_{\tau}(X;F).\]
This is proved in lemma 6.23 of \cite{MVW} for $\tau=\et$ and in lemma 13.4 of \cite{MVW} for $\tau=\Nis$. For $i=0$, $\Hom(R_{\textup{tr}}(X),F)=F(X)$ by the Yoneda lemma. For $i>0$, it suffices to prove that if $F$ is an injective $\tau$-sheaf of $R$-modules with transfers then $H^i(X;F)=0$. For the canonical flasque resolution $E^*(F)$ of $F$, we have that the inclusion $F\rightarrow E^0$ splits by the injectivity of $F$. Therefore, $F$ is a direct summand of $E^0$ in $\textup{Shv}^{tr}_{\tau}(k;R)$. Since $H^i_{\tau}(X;F)$ is a direct summand of $H^i_{\tau}(X;E^0)$, it must vanish for $i>0$.\\
\\
We prove the general theorem for the complex $\underline{K}$ using the hypercohomology spectral sequence. Indeed, we have the following morphism of spectral sequences
\[\xymatrix{E_2^{p,q}:=\Ext^p_{\textup{Shv}^{tr}_{\tau}(k;R)}(R_{\textup{tr}}(X),\cal{H}^q(\underline{K})) \ar@{=>}[r] \ar@{->}[d] & \Ext^{p+q}_{\textup{Shv}^{tr}_{\tau}(k;R)}(R_{\textup{tr}}(X),\underline{K}) \ar@{->}[d] \\ E_2^{p,q}:=H^p_{\tau}(X,\cal{H}^q(\underline{K})) \ar@{=>}[r] & \mathbb{H}^{p+q}_{\tau}(X,\underline{K}).}\]
By the result for sheaves, we have an isomorphism on the $E_2$ page, and so an isomorphism on the abutment. The conclusion follows. 
\end{proof}
\begin{remark}
\textup{This is not a sharp lemma. See proposition 2.2.3 of \cite{CDetale} in which the \'etale version is proved for all Noetherian base schemes $S$ \textit{without} cohomological finiteness assumptions. Also, if $S$ is a Noetherian scheme of finite Krull dimension $d$, then its Nisnevich cohomological dimension is bounded above by $d$. Therefore, the lemma is valid at least for Noetherian base schemes of finite Krull dimension.}
\end{remark}
\begin{proposition}\label{emprop}
Suppose $k$ is algebraically closed of exponential characteristic $p$, and suppose $M\in\textup{DM}^{\textup{eff}}(k;\mathbb{Z}[1/p])_{\textup{tor}}$. For every prime $\ell\neq p$, we have isomorphisms
\[\underline{\textup{\'Et}}_{\ell\ast}:\underline{\pi}_{n,0}\left(\textup{EM}(M)\right)(k)\otimes\mathbb{Z}_{\ell}\rightarrow\pi_n(\underline{\textup{\'Et}}_{\ell}(\textup{EM}(M))).\]
\end{proposition}

\begin{proof}
Let $\pi:\Shv_{\infty}^{Nis}(\Sm/k)\rightarrow \Shv_{\infty}^{\et}(\Sm/k)$ be the morphism of topoi. We can then write $\underline{\textup{\'Et}}_{\ell}L_{\mathbb{A}^1}=L^{H\mathbb{Z}/\ell}\Gamma^{\mathbb{S}}_!\pi^*$. We show that
\[\underline{\textup{\'Et}}_{\ell\ast}:\underline{\pi}_{n,0}\left(\textup{EM}(M)\right)(k)\otimes\mathbb{Z}_{\ell}\rightarrow\pi_n(\underline{\textup{\'Et}}_{\ell}(\textup{EM}(M)))\]
is an isomorphism by showing that
\[\pi^*:\underline{\pi}_{n,0}\left(\textup{EM}(M)\right)(k)\otimes\mathbb{Z}_{\ell}\rightarrow \underline{\pi}^{\et}_{n,0}\left(\pi^*\textup{EM}(M)\right)(k)\otimes\mathbb{Z}_{\ell}\]
and
\[L^{H\mathbb{Z}/\ell}\Gamma^{\mathbb{S}}_!:\underline{\pi}^{\et}_{n,0}\left(\pi^*\textup{EM}(M)\right)(k)\otimes\mathbb{Z}_{\ell}\rightarrow \pi_n(\underline{\textup{\'Et}}_{\ell}(\textup{EM}(M)))\]
are isomorphisms. Here, $\underline{\pi}_{a,b}^{\et}$ is defined as in $\underline{\pi}_{a,b}$ but with the Nisnevich topology replaced by the \'etale topology. We show that the two morphisms above are isomorphisms by showing that the morphisms
\[\Map_{\Spt(k)}(\mathbf{1}_k,\textup{EM}(M))\xrightarrow{\pi^*} \Map_{\Spt_{\et}(k)}(\pi^*\mathbf{1}_k,\pi^*\textup{EM}(M))\]
and
\[\Map_{\Spt_{\et}(k)}(\pi^*\mathbf{1}_k,\pi^*\textup{EM}(M))\xrightarrow{L^{H\mathbb{Z}/\ell}\Gamma^{\mathbb{S}}_!} \Map_{\Pro(\Sp)^{H\mathbb{Z}/\ell}}(L^{H\mathbb{Z}/\ell}\Gamma^{\mathbb{S}}_!\pi^*\mathbf{1}_k,L^{H\mathbb{Z}/\ell}\Gamma^{\mathbb{S}}_!\pi^*\textup{EM}(M))\]
are weak equivalences. By a density argument, the weak equivalence of $\pi^*$ need only be checked for $M=\mathbb{Z}[1/p]^{tr}(X)[s]/N$ for all $X$ smooth projective $k$-scheme, and all $N>1$ coprime to $p$. Remember that $\textup{EM}$ preserves small colimits. Indeed, $\mathbb{Z}[1/p]^{tr}(X)[n]/N$, where $n\in\mathbb{Z}, N>1$ coprime to $p$, and $X$ smooth projective $k$-scheme, compactly generate $\textup{DM}^{\textup{eff}}(k;\mathbb{Z}[1/p])_{\textup{tor}}$ (see proposition 5.5.3 of \cite{Kelly}).

Note that $\pi^*$ is equivalent to the morphism
\[\Map_{\Mod_{M\mathbb{Z}/N}}(M\mathbb{Z}/N,M\mathbb{Z}/N\wedge X_+)\xrightarrow{\pi^*} \Map_{\Mod_{M_{\et}\mathbb{Z}/N}}(\pi^*M\mathbb{Z}/N,\pi^*(M\mathbb{Z}/N\wedge X_+)).\]
The equivalences of tensor triangulated categories $\textup{DM}(k;\mathbb{Z}/N)\simeq\Mod_{M\mathbb{Z}/N}$ and $\textup{DM}_{\et}(k;\mathbb{Z}/N)\simeq\Mod_{M_{\et}\mathbb{Z}/N}$ induce the commutative square
\[\xymatrix{\Map_{\Mod_{M\mathbb{Z}/N}}(M\mathbb{Z}/N,M\mathbb{Z}/N\wedge X_+)\ar@{->}[r]^{\pi^*} \ar@{->}[d]_{\sim} & \Map_{\Mod_{M_{\et}\mathbb{Z}/N}}(\pi^*M\mathbb{Z}/N,\pi^*(M\mathbb{Z}/N\wedge X_+)) \ar@{->}[d]^{\sim} \\ \Map_{\textup{DM}(k;\mathbb{Z}/N)}((\mathbb{Z}/N)^{tr},(\mathbb{Z}/N)^{tr}(X))\ar@{->}[r]^{SV} & \Map_{\textup{DM}_{\et}(k;\mathbb{Z}/N)}((\mathbb{Z}/N)^{tr},(\mathbb{Z}/N)^{tr}(X)).}\]

For a Grothendieck topology $\tau$, recall that $D_{\tau}:=D(\textup{Shv}^{tr}_{\tau}(k;\mathbb{Z}/N))$. Consider the following commutative diagram
\begingroup
    \fontsize{9pt}{11pt}\selectfont
\[\xymatrix{\mathbb{H}_{Nis}^n(\Spec\ k,C_*(\mathbb{Z}/N)^{tr}(X)) \ar@{->}[d]^{=}_{a_{\textup{\'et}}} \ar@{->}[r]^{\sim} & \textup{Ext}^n_{D_{Nis}}((\mathbb{Z}/N)^{tr},C_*(\mathbb{Z}/N)^{tr}(X)) \ar@{->}[d]_{=} \ar@{->}[r]^{\sim}& \pi_{-n}\textup{Map}_{\textup{DM}(k;\mathbb{Z}/N)}((\mathbb{Z}/N)^{tr},(\mathbb{Z}/N)^{tr}(X)) \ar@{->}[d]_{\pi_{-n}SV}\\
\mathbb{H}_{\textup{\'et}}^n(\Spec\ k,C_*(\mathbb{Z}/N)^{tr}(X)) \ar@{->}[r]^{\sim} & \textup{Ext}^n_{D_{\textup{\'et}}}((\mathbb{Z}/N)^{tr},C_*(\mathbb{Z}/N)^{tr}(X)) \ar@{->}[r]^{\sim} & \pi_{-n}\textup{Map}_{\textup{DM}_{\textup{\'et}}(k;\mathbb{Z}/N)}((\mathbb{Z}/N)^{tr},(\mathbb{Z}/N)^{tr}(X))}\]
\endgroup
Since $C_*(\mathbb{Z}/N)^{tr}(X)$ is $\mathbb{A}^1$-local and $\textup{DM}^{\textup{eff}}(k;\mathbb{Z}/N)$ and $\textup{DM}_{\et}^{\textup{eff}}(k;\mathbb{Z}/N)$ embed fully faithfully into $\textup{DM}(k;\mathbb{Z}/N)$ and $\textup{DM}_{\et}(k;\mathbb{Z}/N)$, respectively, the upper and lower right morphisms are isomorphisms. Furthermore, both upper and lower left morphisms are also isomorphism by lemma~\ref{coho}. Since $H^i_{\tau}(\Spec k;-)$ vanishes for $i>0$, $H^0_{\tau}(\Spec k;-)$ is an exact functor. Therefore, $\mathbb{H}^*_{\tau}(\Spec k;\underline{K})=H^*(\underline{K}(\Spec k))$. As a result, the left vertical map is also an isomorphism. We conclude that $SV$ is a weak equivalence. By the commutative square above relating $\pi^*$ to $SV$, we deduce that $\pi^*$ is a weak equivalence.\\
\\
We use another density argument to prove that 
\[L^{H\mathbb{Z}/\ell}\Gamma^{\mathbb{S}}_!:\underline{\pi}^{\et}_{n,0}\left(\pi^*\textup{EM}(M)\right)(k)\otimes\mathbb{Z}_{\ell}\rightarrow \pi_n(\underline{\textup{\'Et}}_{\ell}(\textup{EM}(M)))\]
is an isomorphism. For $M=\mathbb{Z}[1/p]^{tr}(X)/N$ as before, $\pi^*M\in\textup{DM}_{\et}^{\textup{eff}}(k;\mathbb{Z}/N)\simeq\Mod_{H\mathbb{Z}/N}$. Since this category is compactly generated by shifts of the constant sheaf $\mathbb{Z}/N$, and since our functors are exact functors preserving colimits, it suffices to check that 
\[\Map_{\Spt_{\et}(k)}(\pi^*\mathbf{1}_k,\pi^*M\mathbb{Z}/\ell^m)\xrightarrow{L^{H\mathbb{Z}/\ell}\Gamma^{\mathbb{S}}_!} \Map_{\Pro(\Sp)^{H\mathbb{Z}/\ell}}(L^{H\mathbb{Z}/\ell}\Gamma^{\mathbb{S}}_!\pi^*\mathbf{1}_k,L^{H\mathbb{Z}/\ell}\Gamma^{\mathbb{S}}_!\pi^*(M\mathbb{Z}/\ell^m))\]
is a weak equivalence. However, note that $\pi^*(M\mathbb{Z}/\ell^m)\simeq H\mathbb{Z}/\ell^m$ is in the image of the constant functor $c:\Sp\rightarrow\Shv_{\infty}^{\et}(\Sm/k,\Sp)$. Since $L^{H\mathbb{Z}/\ell}\Gamma^{\mathbb{S}}_!$ is a pro-left adjoint to $c$, the above is formally a weak equivalence. 
\end{proof}

In order to prove the next proposition, we need the following lemma.

\begin{lemma}\label{cruc}
\[\underline{\textup{\'Et}}_{\ell}(M\mathbb{Z}[1/p])\simeq H\mathbb{Z}_{\ell}\]
is a weak equivalence in $\Pro(\cal{S})^{\mathbb{Z}/\ell}$.
\end{lemma}

\begin{proof}
Note that
\[M\mathbb{Z}[1/p]\simeq\colim_{n\rightarrow\infty}\Sigma_{\mathbb{P}^1}^{-n}\Sigma^{\infty}_{\mathbb{P}^1}K(\mathbb{Z}[1/p](n),2n).\]
Since $\underline{\textup{\'Et}}_{\ell}$ preserves colimits, we obtain
\begin{eqnarray*}\underline{\textup{\'Et}}_{\ell}(M\mathbb{Z}[1/p]) &\simeq& \colim_{n\rightarrow{\infty}}\underline{\textup{\'Et}}_{\ell}(\Sigma_{\mathbb{P}^1}^{-n}\Sigma^{\infty}_{\mathbb{P}^1}K(\mathbb{Z}[1/p](n),2n))\\ &\simeq& \colim_{n\rightarrow{\infty}}\Sigma_{S_{\ell}^2}^{-n}\Sigma^{\infty}_{S^2_{\ell}}\textup{\'Et}_{\ell}(K(\mathbb{Z}[1/p](n),2n))\\ &\simeq^*& \colim_{n\rightarrow{\infty}}\Sigma_{S_{\ell}^2}^{-n}\Sigma^{\infty}_{S^2_{\ell}}K(\mathbb{Z}_{\ell},2n)\\ &\simeq& L^{H\mathbb{Z}/\ell}\colim_{n\rightarrow{\infty}}\Sigma_{S^2}^{-n}\Sigma^{\infty}_{S^2}K(\mathbb{Z},2n) \\ &\simeq& L^{H\mathbb{Z}/\ell}H\mathbb{Z}\simeq H\mathbb{Z}_{\ell},
\end{eqnarray*}
where the equivalence $\simeq^*$ follows from corollary 8.5 of \cite{Hoyois}.
\end{proof}

\begin{proposition}\label{sqprop}
For every $q$ and $n$, the map 
\[\underline{\textup{\'Et}}_{\ell\ast}:\underline{\pi}_{n,0}\left(s_q(\mathbf{1}_k)\right)(k)\otimes\mathbb{Z}_{\ell}\rightarrow\pi_n\left(\underline{\textup{\'Et}}_{\ell}(s_q(\mathbf{1}_k))\right)\]
is an isomorphism.
\end{proposition}

\begin{proof}
Note that $\mathbf{1}_k\in\Spt^{\textup{eff}}(k)$, hence for $q<0$, $s_q(\mathbf{1})=0$. For $q=0$, $s_0(\mathbf{1})\simeq M\mathbb{Z}$ by \cite{HConiveau}. As a result of this and the previous lemma, $\underline{\textup{\'Et}}_{\ell}(s_0(\mathbf{1}))=\underline{\textup{\'Et}}_{\ell}(M\mathbb{Z}[1/p])\simeq H\mathbb{Z}_{\ell}$ in $\Pro(\Sp)^{H\mathbb{Z}/\ell}$. Therefore,
\[\underline{\pi}_{n,0}(s_0(\mathbf{1}))(k)=\pi_{n,0}(M\mathbb{Z})=H_n^{\textup{Sus}}(k;\mathbb{Z})=\begin{cases} 0\ \text{for}\ n\neq 0\\ \mathbb{Z}\ \text{for}\ n=0.\end{cases}\]
On the other hand, 
\[\pi_n(H\mathbb{Z}_{\ell})=\begin{cases} 0\ \text{for}\ n\neq 0\\ \mathbb{Z}_{\ell}\ \text{for}\ n=0.\end{cases}\]
The unit in $\underline{\pi}_{0,0}(M\mathbb{Z}[1/p])(k)\otimes\mathbb{Z}_{\ell}$ is induced by $\mathbf{1}\rightarrow M\mathbb{Z}[1/p]$ which goes to the unit $L^{H\mathbb{Z}/\ell}\mathbb{S}\rightarrow H\mathbb{Z}_{\ell}$ in $\pi_0(H\mathbb{Z}_{\ell})$ under the stable $\ell$-adic \'etale realization. Therefore, we have the isomorphism when $q\leq 0$. For $q>0$, $\pi_q^{\mu}(\mathbf{1})$ is in $\textup{DM}^{\textup{eff}}(k)_{\textup{tor}}$ by lemma~\ref{2}. Since $s_q(\mathbf{1})=\textup{EM}(\pi_q^{\mu}(q)[2q])$, we have the isomorphism by proposition~\ref{emprop}. The conclusion follows.
\end{proof}
Using this proposition, a comparison of the two spectral sequences of the previous section gives us the main theorem.
\begin{theorem}
Let $k$ be an algebraically closed field of exponential characteristic $p$. Then for every prime $\ell\neq p$, the stable $\ell$-adic \'etale realization induces an isomorphism
\[\underline{\textup{\'Et}}_{\ell\ast}:\underline{\pi}_{*,0}(\mathbf{1}_k)(k)\otimes_{\mathbb{Z}}\mathbb{Z}_{\ell}\rightarrow \pi_*\left(\mathbb{S}\right)\otimes_{\mathbb{Z}}\mathbb{Z}_{\ell}\]
\end{theorem}

\begin{proof}
Consider the case $n=0$. By Morel's theorem, $\underline{\pi}_{0,0}\left(\mathbf{1}[1/p]\right)(k)=\textup{GW}(k)[1/p]\simeq\mathbb{Z}[1/p]$ via the dimension function lemma 3.10 of \cite{Morel}. Note that here we use the fact that $k$ is algebraically closed. As a result, $\mathbf{1}[1/p]\rightarrow s_0(\mathbf{1})[1/p]\simeq M\mathbb{Z}[1/p]$ induces the natural map
\[\underline{\pi}_{0,0}\left(\mathbf{1}[1/p]\right)(k)\rightarrow\underline{\pi}_{0,0}(M\mathbb{Z}[1/p])(k)\simeq\mathbb{Z}[1/p]\]
that induces an isomorphism
\[\underline{\pi}_{0,0}(\mathbf{1})(k)\otimes\mathbb{Z}_{\ell}\xrightarrow{\sim}\underline{\pi}_{0,0}(M\mathbb{Z}[1/p])(k)\otimes\mathbb{Z}_{\ell}.\]
On the other hand the first slice of $L^{H\mathbb{Z}/\ell}\mathbb{S}$ in $\Pro(\Sp)^{H\mathbb{Z}/\ell}$ is $L^{H\mathbb{Z}/\ell}\mathbb{S}\rightarrow H\mathbb{Z}_{\ell}$. However, the image of $\mathbf{1}\rightarrow M\mathbb{Z}$ under $\underline{\textup{\'Et}}_{\ell}$ is the generator $L^{H\mathbb{Z}/\ell}\mathbb{S}\rightarrow H\mathbb{Z}_{\ell}$. We therefore have the commutative square
\[\xymatrix{\underline{\pi}_{0,0}(\mathbf{1})(k)\otimes\mathbb{Z}_{\ell} \ar@{->}[r]^{\sim} \ar@{->}[d]_{\underline{\textup{\'Et}}_{\ell\ast}} & \underline{\pi}_{0,0}(s_0(\mathbf{1}))(k)\otimes\mathbb{Z}_{\ell} \ar@{->}[d]^{\sim}_{\underline{\textup{\'Et}}_{\ell\ast}}\\ \pi_0(\mathbb{S})\otimes\mathbb{Z}_{\ell} \ar@{->}[r]^{\sim} & \pi_0(H\mathbb{Z}_{\ell}).}\]
From this it follows that
\[\underline{\textup{\'Et}}_{\ell\ast}:\underline{\pi}_{0,0}(\mathbf{1})(k)\otimes_{\mathbb{Z}}\mathbb{Z}_{\ell}\rightarrow \pi_0\left(\mathbb{S}\right)\otimes_{\mathbb{Z}}\mathbb{Z}_{\ell}\]
is an isomorphism. From the slice tower for $\mathbf{1}$ we get the distinguished triangle
\[f_1\mathbf{1}\rightarrow\mathbf{1}\rightarrow s_0\mathbf{1}\rightarrow f_1\mathbf{1}[1].\]
Using $s_0(\mathbf{1})\simeq M\mathbb{Z}$, we obtain the long exact sequence of stable motivic homotopy groups
\[\hdots\rightarrow\underline{\pi}_{a+1,0}(M\mathbb{Z}[1/p])(k)\rightarrow\underline{\pi}_{a,0}(f_1(\mathbf{1})[1/p])(k)\rightarrow\underline{\pi}_{a,0}(\mathbf{1}[1/p])(k)\rightarrow\underline{\pi}_{a,0}(M\mathbb{Z}[1/p])(k)\rightarrow\hdots\]
However, $\underline{\pi}_{a,0}(M\mathbb{Z}[1/p])(k)=H^{-a,0}(k;\mathbb{Z}[1/p])$ is $0$ if $a\neq 0$, and is $\mathbb{Z}[1/p]$ if $a=0$. By lemmas 4.4 and 4.5 of \cite{Levine}, and proposition 4.7 of \cite{Levine}, $f_1(\mathbf{1})$ is topologically $(-1)$-connected, and so $f_1(\mathbf{1})[1/p]$ is also topologically $(-1)$-connected. As a result, the maps\[\underline{\pi}_{a,0}\left(f_1(\mathbf{1})[1/p]\right)\rightarrow\underline{\pi}_{a,0}\left(\mathbf{1}[1/p]\right)\]
are isomorphisms for $a\neq 0$. Also, $\underline{\pi}_{0,0}\left(f_1(\mathbf{1})[1/p]\right)=0$. $\underline{\textup{\'Et}}_{\ell}$ produces the distinguished triangle
\[\underline{\textup{\'Et}}_{\ell}\left(f_1\mathbf{1}\right)\rightarrow\underline{\textup{\'Et}}_{\ell}\left(\mathbf{1}\right)\rightarrow \underline{\textup{\'Et}}_{\ell}\left(s_0\mathbf{1}\right)\rightarrow \underline{\textup{\'Et}}_{\ell}\left(f_1\mathbf{1}\right)[1]\]
from which we get the long exact sequence of homotopy groups
\[\hdots\rightarrow\pi_{a+1}(\underline{\textup{\'Et}}_{\ell}\left(M\mathbb{Z}\right))\rightarrow\pi_a(\underline{\textup{\'Et}}_{\ell}\left(f_1(\mathbf{1}))\right))\rightarrow\pi_a(\underline{\textup{\'Et}}_{\ell}\left(\mathbf{1}\right))\rightarrow\pi_a(\underline{\textup{\'Et}}_{\ell}\left(M\mathbb{Z}\right))\rightarrow\hdots\]
We know by proposition~\ref{etaleconnectivity} that $\underline{\textup{\'Et}}_{\ell}\left(f_1(\mathbf{1})\right)$ is $0$-connected. In combination with the long exact sequence of stable motivic homotopy groups above, we obtain commutative squares
\[\xymatrix{\underline{\pi}_{a,0}\left(f_1(\mathbf{1})\right)(k)\otimes\mathbb{Z}_{\ell} \ar@{->}[r]^{\sim} \ar@{->}[d]_{\underline{\textup{\'Et}}_{\ell\ast}} & \underline{\pi}_{a,0}\left(\mathbf{1}\right)(k)\otimes\mathbb{Z}_{\ell} \ar@{->}[d]_{\underline{\textup{\'Et}}_{\ell\ast}}\\ \pi_a(\underline{\textup{\'Et}}_{\ell}\left(f_1(\mathbf{1})\right)\otimes\mathbb{Z}_{\ell} \ar@{->}[r]^{\sim} & \pi_a(\underline{\textup{\'Et}}_{\ell}(\mathbf{1})).}\]
As a result, it suffices to prove that
\[\underline{\textup{\'Et}}_{\ell\ast}:\underline{\pi}_{a,0}\left(f_1(\mathbf{1})\right)\otimes\mathbb{Z}_{\ell}\rightarrow\pi_a\left(\underline{\textup{\'Et}}_{\ell}\left(f_1(\mathbf{1})\right)\right)\]
is an isomorphism. We prove this via a comparison of spectral sequences. Recall that we have a morphism
\[\xymatrix{
_{I}E^{p,q}_1:=\underline{\pi}_{p+q,0}(s_q(\mathbf{1}))(k)\otimes\mathbb{Z}_{\ell} \ar@{=>}[r] \ar@{->}[d]_{\underline{\textup{\'Et}}_{\ell\ast}} & \underline{\pi}_{p+q,0}(f_1(\mathbf{1}))(k)\otimes\mathbb{Z}_{\ell} \ar@{->}[d]^{\underline{\textup{\'Et}}_{\ell\ast}}\\ _{II}E^{p,q}_1:=\pi_{p+q}\left(\underline{\textup{\'Et}}_{\ell}\left(s_q(\mathbf{1})\right)\right) \ar@{=>}[r] & \pi_{p+q}\left(\underline{\textup{\'Et}}_{\ell}\left(f_1(\mathbf{1})\right)\right)
}\]
of strongly convergent spectral sequences. By proposition~\ref{sqprop}, $\underline{\textup{\'Et}}_{\ell\ast}$ is an isomorphism on the $E_1$-page, and so it is an isomorphism on the abutment. The conclusion follows. 
\end{proof}

\section{Generalized \'Etale Suslin-Voevodsky Theorem}\label{susvoe}
The Suslin-Voevodsky theorem gives an isomorphism between Suslin (co)homology and singular or \'etale (co)homology, all with $\mathbb{Z}/N$ coefficient with $(N,p)=1$ ($p$ the exponential characteristic of $k$), of a finite type separated $k$-scheme, $k$ algebraically closed. In theorem 7.1 of \cite{Levine}, Levine proved a homotopy theoretic generalization of the Suslin-Voevodsky theorem relating motivic cohomology to singular cohomology. Here, we prove a homotopy theoretic generalization of the \'etale variant of this theorem.

\begin{theorem}\label{susvoe}
Suppose $k$ is an algebraically closed field of exponential characteristic $p$, $E\in\Spt^{\textup{eff}}(k)_{\textup{tor}}$ is an effective torsion $\mathbb{P}^1$-spectrum, and $\ell\neq p$ is a prime. Then
\[\underline{\textup{\'Et}}_{\ell\ast}:\underline{\pi}_{n,0}(E)(k)\otimes\mathbb{Z}_{\ell}\rightarrow\pi_n\left(\underline{\textup{\'Et}}_{\ell}(E)\right)\]
is an isomorphism.
\end{theorem}

Before proving this lemma, let us justify why this is a generalization of the usual Suslin-Voevodsky theorem. We need the following lemma.

\begin{lemma}\label{fincor}
Given a prime number $\ell$ different from the exponential characteristic $p$ of $k$, and given a smooth $k$-scheme $X$,\[\underline{\textup{\'Et}}_{\ell}(M\mathbb{Z}/N\wedge X_+)\rightarrow \underline{\textup{\'Et}}_{\ell}(M\mathbb{Z}/N)\wedge \textup{\'Et}_{\ell}X_+\simeq H\mathbb{Z}/\ell^m\wedge \textup{\'Et}_{\ell}X_+\]
is a weak equivalence for every positive integer $N>1$ coprime to $p$ with $\ell$-adic valuation $m$.
\end{lemma}

\begin{proof}
Note that
\[M\mathbb{Z}/N\wedge X_+\simeq\colim_{n\rightarrow\infty}\Sigma_{\mathbb{P}^1}^{-n}\Sigma^{\infty}_{\mathbb{P}^1}K(\mathbb{Z}/N(n),2n)\wedge X_+.\]
Since $\underline{\textup{\'Et}}_{\ell}$ preserves colimits, we obtain
\begin{eqnarray*}\underline{\textup{\'Et}}_{\ell}(M\mathbb{Z}/N\wedge X_+) &\simeq& \colim_{n\rightarrow{\infty}}\underline{\textup{\'Et}}_{\ell}(\Sigma_{\mathbb{P}^1}^{-n}\Sigma^{\infty}_{\mathbb{P}^1}K(\mathbb{Z}/\ell^m(n),2n)\wedge (\Sigma_{\mathbb{P}^1}^{\infty}X_+/\ell^m))\\ &\simeq& \colim_{n\rightarrow{\infty}}(\Sigma_{S_{\ell}^2}^{-n}\Sigma^{\infty}_{S^2_{\ell}}\textup{\'Et}_{\ell}K(\mathbb{Z}/\ell^m(n),2n)\wedge \textup{\'Et}_{\ell}X_+/\ell^m)\\ &\simeq^{(*)}& \colim_{n\rightarrow{\infty}}(\Sigma_{S_{\ell}^2}^{-n}\Sigma^{\infty}_{S^2_{\ell}}K(\mathbb{Z}/\ell^m,2n)\wedge \textup{\'Et}_{\ell}X_+/\ell^m)\\ &\simeq^{(**)}& \left(\colim_{n\rightarrow{\infty}}\Sigma_{S^2}^{-n}\Sigma^{\infty}_{S^2}K(\mathbb{Z}/\ell^m,2n)\right)\wedge \textup{\'Et}_{\ell}X_+/\ell^m\\ &\simeq& H\mathbb{Z}/\ell^m\wedge \textup{\'Et}_{\ell}X_+,
\end{eqnarray*}
where the equivalence $(*)$ follows from corollary 8.5 of \cite{Hoyois}, and $(**)$ follows from $\textup{\'Et}_{\ell}X_+/\ell^m$ being a space, not just a pro-space.
\end{proof}

In particular, if $X$ is a smooth $k$-scheme and $E=M\mathbb{Z}/\ell\wedge X_+$ with $\ell\neq p$ a prime, then $\underline{\textup{\'Et}}_{\ell}(E)\simeq H\mathbb{Z}/\ell\wedge\textup{\'Et}_{\ell}X_+$ by lemma~\ref{fincor}. Applying theorem~\ref{susvoe}, we obtain the isomorphism
\[H^{\textup{sus}}_n(X;\mathbb{Z}/\ell)\simeq H^{\textup{\'et}}_n(X;\mathbb{Z}/\ell),\]
from which we get the Suslin-Voevodsky comparison theorem
\[H^n_{\textup{sus}}(X;\mathbb{Z}/\ell)\simeq H^n_{\textup{\'et}}(X;\mathbb{Z}/\ell)\]
via dualization. We now prove the above theorem.

\begin{proof}
$\Spt^{\textup{eff}}(k)_{\textup{tor}}$ is compactly generated by the objects $\{\Sigma^{p,q}\Sigma_{\mathbb{P}^1}^{\infty}X_+/N|q\geq 0,\ p\in\mathbb{Z},\ (N,p)=1,\ X\in\Sm/k\}$. Therefore, it suffices to prove the theorem for $E\in\Spt^{\textup{eff}}_{\textup{fin}}(k)_{\textup{tor}}$. As before, the slice tower of $E$ gives a strongly convergent spectral sequence
\[E_1^{p,q}=\underline{\pi}_{p+q,0}(s_qE)(k)\otimes\mathbb{Z}_{\ell}\implies \underline{\pi}_{p+q,0}(E)(k)\otimes\mathbb{Z}_{\ell}.\]
$E$ is topologically $(N-1)$-connected for some $N$, and so by proposition~\ref{etaleconnectivity}, $\underline{\textup{\'Et}}_{\ell}(f_qE)$ is $(q+N-1)$-connected. Therefore, the \'etale realization of the slice tower gives a strongly convergent spectral sequence
\[E_1^{p,q}=\pi_{p+q}(\underline{\textup{\'Et}}_{\ell}(s_qE[1/p]))\implies \pi_{p+q}(\underline{\textup{\'Et}}_{\ell}(E)).\]
In fact, we have a morphism of spectral sequences from the former to the latter spectral sequence induced by stable $\ell$-adic \'etale realization.\\
\\
Since $E$ is a torsion effective $\mathbb{P}^1$-spectrum, $s_qE\simeq EM(\pi_q^{\mu}(E)(q)[2q])$ for some $\pi_q^{\mu}(E)\in\textup{DM}^{\textup{eff}}(k)_{\textup{tor}}$ by lemma~\ref{torlem}, and so proposition~\ref{emprop} implies that we have an isomorphism of the $E_1$-pages. The conclusion follows.
\end{proof}

\begin{corollary}\label{cor}
Suppose $k$ is an algebraically closed field of characteristic zero endowed with an inclusion $\sigma:k\hookrightarrow\mathbb{C}$. Let $E$ be a effective torsion motivic $\mathbb{P}^1$-spectrum, and let $\ell$ be a prime. Then for each $n$, there is an isomorphism
\[\pi_n(\underline{Et}_{\ell}(E))\cong \pi_n(\textup{Re}^{\sigma}_B(E))\otimes\mathbb{Z}_{\ell},\]
where $\textup{Re}^{\sigma}_B$ is the stable Betti realization functor (see \cite{Levine}).
\end{corollary}

\begin{proof}
Combine theorem~\ref{susvoe} above, and theorem 7.1 of \cite{Levine}.
\end{proof}

\begin{remark}
\textup{If our algebraically closed field $k$ is of characteristic zero, we can replace $\ell$-adic Bousfield localization in our construction of the stable $\ell$-adic \'etale realization with profinite completion. Going through the same proofs and constructions tells us that when $k$ has characteristic zero, there is a refinement of the results in this paper. More precisely, we may replace all $\ell$-completions with pro-finite completions, and all $-\otimes\mathbb{Z}_{\ell}$ with $-\otimes\hat{\mathbb{Z}}$. In particular, we may replace the isomorphism of corollary~\ref{cor} with the isomorphism
\[\pi_n(\underline{\hat{\textup{Et}}}(E))\cong \pi_n(\textup{Re}^{\sigma}_B(E))^{\wedge},\]
where $\underline{\hat{\textup{Et}}}$ is our pro-finitely completed (as apposed to $\ell$-adic-completed) stable \'etale realization functor.}
\end{remark}

\textit{Acknowledgments.} I would like to thank Professors Denis-Charles Cisinski, Peter Oszvath, and Charles Weibel for their support and encouragement. In particular, I would like to thank Professor Weibel for comments on an earlier version of this paper. I am also greatly indebted to Professor Cisinski for patiently discussing the details of this paper, and for suggesting some improvements. I would like to thank Elden Elmanto, Adeel Khan, and Markus Land for useful discussions. Also, thanks to Marc Hoyois for pointing out an error in the proof of a proposition of a previous version of this paper. This project was supported by Princeton University and the University of Regensburg (SFB1085, Higher Invariants).

\small{\textsc{Department of Mathematics, Princeton University, Princeton NJ}}\\
\textit{Email address:} \texttt{\small{mzargar@math.princeton.edu}}
\end{document}